\documentclass[12pt]{article}

\usepackage{amsfonts}

\usepackage{amsmath,amssymb,amstext,amsthm,amscd,mathtools,graphicx,float,subcaption,caption,todonotes,tikz,tikz-cd,hyperref}

\usetikzlibrary{matrix,arrows,decorations.pathmorphing}

\theoremstyle{plain}
\newtheorem{thm}{Theorem}[section] 
\newtheorem{pro}[thm]{Proposition}
\newtheorem{lem}[thm]{Lemma}

\newtheorem{con}{Conjecture}

\theoremstyle{definition}

\theoremstyle{remark}
\newtheorem*{rem}{Remark}

\def\ds{\displaystyle}
\def\pf{{\em Proof.}\ \,}

\def\su{\subseteq}

\def\({\left(}
\def\){\right)}

\def\<{\langle}
\def\>{\rangle}

\def\lfun{\longrightarrow}
\def\sur{\twoheadrightarrow}

\def\bC{{\mathbb{C}}}
\def\bE{{\mathbb{E}}}
\def\bF{{\mathbb{F}}}
\def\bH{{\mathbb{H}}}

\def\bP{{\mathbb{P}}}

\def\bR{{\mathbb{R}}}

\def\bZ{{\mathbb{Z}}}

\def\cA{{\mathcal A}}

\def\cF{{\mathcal F}}

\def\cL{{\mathcal L}}
\def\cM{{\mathcal M}}

\def\cO{{\mathcal O}}

\def\cS{{\mathcal S}}
\def\cT{{\mathcal T}}

\def\cW{{\mathcal W}}

\def\ch{{\textrm{ch}}}

\def\Coh{{\textrm{Coh\,}}}

\def\Im{{\textrm{Im\,}}}

\def\Real{{\textrm{Re\,}}}
\def\rk{{\textrm{rk}}}

\newcommand{\Proj}{{\mathbb P}^1}
\newcommand{\PP}{{\mathbb P}^1 \times {\mathbb P}^1}

\newcommand{\Ptoo}{{\mathbb P}^2}

\newcommand{\stabdivS}{\text{Stab}_{div}(S)}

\newcommand{\Kl}{\underline{K}}

\newcommand{\Jl}{\overline{J}}

\newcommand{\TE}{{\cO_S(E)|_E}}

\newcommand{\Homdot}[2]{\text{Hom}^\bullet (#1,#2)}
\newcommand{\Homx}[3]{\text{Hom}^{#1} (#2,#3)}
\newcommand{\Emo}{\bE_{-1}}

\newcommand{\To}{\rightarrow}

\begin{document}

\title{Projectivity of Bridgeland Moduli Spaces on Del Pezzo Surfaces of Picard Rank 2}


\author{Daniele Arcara and Eric Miles}
\date{}

\maketitle

\begin{abstract}

We prove that, for a natural class of Bridgeland stability conditions on $\PP$ and the blow-up of $\bP^2$ at a point, the moduli spaces of Bridgeland semistable objects are projective. Our technique is to find suitable regions of stability conditions with hearts that are (after ``rotation'') equivalent to representations of a quiver. The helix and tilting theory is well-behaved on Del Pezzo surfaces and we conjecture  that this program (begun in \cite{ABCH}) runs successfully for all Del Pezzo surfaces, and the analogous Bridgeland moduli spaces are projective.


\end{abstract} 

\tableofcontents



\section{Introduction}

Let $S$ be a smooth projective complex surface and $D(S) = D^b(\Coh S)$ be its associated bounded derived category. A Bridgeland stability condition $\sigma$ gives a notion of stability for objects of $D(S)$ and it is interesting to study the moduli space $\cM_\sigma(v)$ of $\sigma$-semistable objects of Chern character $v$.

Unlike Mumford and Gieseker stability, Bridgeland stability conditions are not a priori tied to Geometric Invariant Theory (GIT), and consequently the structure of Bridgeland moduli spaces $\cM_\sigma(v)$ is not fully understood in general, e.g., are they projective? A general approach to this question is taken in \cite{BMprojandbirat} where Bayer and Macr\`i construct nef divisors on Bridgeland moduli spaces and show that for K3 surfaces, the divisors are ample.


Another method to deduce structure for these moduli spaces is to find particular stability conditions that have ties to GIT, and exploit this connection. For example, the first author and Bertram \cite{ABL} define a class of stability condition defined by a choice of general and ample divisor in $S$. When $S$ has Picard rank 1, the space of these stability conditions is parameterized by real variables, $s,t$ with $t>0$, and for large $t$ Bridgeland stability is equivalent to Gieseker stability (e.g. \cite[Proposition 6.2]{ABCH}) and thus the corresponding Bridgeland moduli spaces have a GIT interpretation and are projective. 

For $S$ a principally polarized abelian surface, Maciocia and Meachan \cite{MacMea} use this connection to deduce projectivity for other Bridgeland stability conditions by ``sliding down the wall'': for invariants corresponding to twisted ideal sheaves, they show that one can move a stability condition (taken from a 1-parameter family of stability conditions) so that $t$ is arbitrarily small without crossing any walls for those invariants (the walls are nested semi-circles centered on the $s$-axis). They conclude projectivity by relating stability conditions with small $t$ to those with large $t$ via a Fourier-Mukai transform.

For $S=\Ptoo$, the first author, Bertram, Coskun, and Huizenga  \cite{ABCH} find stability conditions with small $t$ where Bridgeland stability for complexes of sheaves is equivalent to King's $\chi$-stability for representations of a quiver \cite{kingquiver}. By a result of King, it follows that the associated Bridgeland moduli spaces are projective. To extend this structure to other stability conditions, they show that for invariants satisfying the Bogomolov inequality, one can ``slide down the wall'' and see the original Bridgeland moduli space as isomorphic to one associated to a stability condition with small $t$.


The quivers involved in \cite{ABCH} are associated to certain exceptional collections of objects in $D^b(\bP^2)$. These exceptional collections exist on other surfaces, e.g., Del Pezzo surfaces, but the ``quiver regions'' associated to the most natural exceptional collections are too small. Furthermore, determining new exceptional collections is difficult, as the action of the Artin braid group on the set of exceptional collections for $\Ptoo$ does not apply when $\dim K(S) \otimes \bC > 1 + \dim S$, which is the case for all other Del Pezzos.

However, Bridgeland and Stern \cite{helicesondelpezzos} describe an operation (called mutation defined by a height function) which produces new exceptional collections in this more general setting (and matches the action of the braid group when $S=\Ptoo$). This operation is applicable for any smooth Fano variety, but works in an ideal way on Del Pezzo surfaces. We look to extend the techniques of \cite{ABCH} to the other Del Pezzo surfaces.


\begin{con}
\label{con:main}
The program of \cite{ABCH} can be carried out on any Del Pezzo surface $S$, yielding the projectivity of the spaces $\cM_\sigma(v)$.
\end{con}
 
For the Del Pezzo surfaces of Picard rank 2, $S=\PP$ and $S=Bl_p\Ptoo$, we find stability conditions that, after an operation called ``rotation,'' have their hearts equivalent to a category of quiver representations. These hearts are generated by (shifts of) line bundles for $S=\PP$ and (shifts of) line bundles and a torsion sheaf for $S=Bl_p\Ptoo$. Understanding the Bridgeland stability of these objects is essential to finding these ``quiver stability conditions,'' and while the stability of line bundles is fully understood for these surfaces \cite{linebundlessurfs}, a full characterization of their stability is not known when the surface has Picard rank greater than 2.

Following \cite{ABCH}, Bridgeland stability for these quiver stability conditions is  equivalent to King's stability of representations of a quiver and we deduce projectivity for the associated moduli spaces of Bridgeland semi-stable objects.
To extend this structure to moduli spaces associated to other stability conditions we ``slide down the wall,'' using a result of Maciocia on the nestedness of walls in certain vertical planes in the space of stability conditions. 

Our main theorem is the following.




\begin{thm}
The conjecture is true for $S=\bP^1\times\bP^1$ and $S=Bl_p\bP^2$, which are Del Pezzo surfaces of Picard Rank $2$.
In particular, every moduli space $\cM_\sigma(v)$ is equivalent to a moduli space of representations of a quiver, and is therefore projective.
\end{thm}

The paper is organized as follows.
In Section \ref{alltools}, we recall some of the definitions and tools that are needed in the rest of the paper.
In Section \ref{strategy} we describe our strategy.
In Section \ref{proofP1xP1} we prove the conjecture for $S=\bP^1\times\bP^1$.
Finally, in Section \ref{proofblpP2}, we prove the conjecture for $S=Bl_p\bP^2$.





\section{Useful definitions}\label{alltools}

In this section, we recall definitions and tools that we need in the rest of the paper.

\subsection{Bridgeland stability conditions on surfaces}

For $S$ a smooth projective complex surface, \cite{ABL} define a natural class of Bridgeland stability conditions that depend on a choice of ample and general divisor. We denote the set of such Bridgeland stability conditions $\stabdivS$ (the ``div'' stands for ``divisor''). We omit a general introduction to Bridgeland stability conditions and instead point out where the \cite{ABL} construction meets the requirements ``as we go.'' For full generality, the interested reader should see \cite{stabcondsontricats}.

Let $S$ be a smooth projective surface. Given two $\bR$-divisors $D,H$ with $H$ ample, we define a stability condition $\sigma_{D,H}$. To do so, we must specify a heart $\cA_{D,H}$ of a bounded $t$-structure on $D(S)$ and a function $Z_{D,H}:K_{num}(S) \to \bC$ (called the \textit{central charge}) satisfying certain positivity, filtration, and non-degenerate conditions. 

Our heart is generated by torsion sheaves, Mumford $H$-stable sheaves of ``high slope,'' and shifts of Mumford $H$-stable sheaves of ``low slope,'' where the Mumford $H$-slope is
$$ \mu_H(G) = \ds\frac{c_1(G).H}{\rk(G) H^2}. $$

Specifically, let $\cA_{D,H}$ be the tilt of the standard $t$-structure on $D(S)$ at $\mu_H(D) =\ds\frac{D.H}{H^2}$ defined by $\cA_{D,H} = \{G\in D(S) \mid H^i(G)=0 \text{ for } i\neq-1,0,\ H^{-1}(G)\in\cF_{D,H},\ H^0(G)\in\cT_{D,H}\}$ where
\begin{itemize}
\item
$\cT_{D,H}\subset \Coh(S)$ is the full subcategory closed under extensions generated by torsion sheaves and $\mu_H$-stable sheaves $G$ with $\mu_H(G)>\ds\frac{D.H}{H^2}$.
\item
$\cF_{D,H}\subset\Coh(S)$ is the full subcategory closed under extensions generated by $\mu_H$-stable sheaves $G$ with $\mu_H(G)\leq\ds\frac{D.H}{H^2}$.
\end{itemize}
\vspace{.1in}

Now define $Z_{D,H}$ by
$ Z_{D,H}(G) = - \ds\int{e^{-(D+iH)}\ch(G)}. $
It is equal to
$$ Z_{D,H}(G) = \left( -\ch_2(G) + c_1(G).D - \frac{\rk(G)}{2}(D^2 - H^2) \right)+ i \left( c_1(G).H - \rk(G) D.H \right) $$

The central charge $Z_{D,H}$ satisfies the following positivity property: for $0\neq G \in \cA_{D,H}$ we have $\Im Z(G) \geq 0$ and if $\Im Z(G) = 0$ then $\Real Z(G) < 0$. This property allows us to define stability for an object $G\in \cA_{D,H}$ using the ``Bridgeland slope'' function $$ \beta(G) = -\frac{\Real Z(G)}{\Im Z(G)} \in (-\infty, \infty] $$
(Note that if $Z(G)=-1$ then $\beta(G) = \infty$, and if $Z(G)=\sqrt{-1}$ then $\beta(G) = 0$.) We say that $G\in\cA_{D,H}$ is $\sigma_{D,H}$-stable (resp.\ $\sigma_{D,H}$-semistable) if for all nontrivial $G'\subset G$ in $\cA_{D,H}$ we have $\beta(G)>\beta(G')$ (resp.\ $\beta(G)\geq\beta(G')$). We extend the notion of stability to all objects of the derived category: $G\in D(S)$ is $\sigma_{D,H}$-(semi)stable if some shift $G[k]\in\cA_{D,H}$ and $G[k]$ is $\sigma_{D,H}$-(semi)stable.

There exist Harder-Narasimhan filtrations of objects $G\in\cA_{D,H}$ with respect to $\sigma_{D,H}$-semistable objects, defined analogously to Harder-Narasimhan filtrations of coherent sheaves with respect to Mumford $H$-semistable sheaves. By \cite[Corollary 2.1]{ABL} and \cite[Sections 3.6 \& 3.7]{stabcondsextrcontrs}, $\sigma_{D,H}$ is a full, numerical stability condition on $S$.

\subsection{Exceptional collections and associated hearts}
\label{exccolhts}

Here we describe the heart and quiver associated to a full, strong exceptional collection, as well as an operation that (in certain circumstances) yields new collections from others. We use $E_\star$ instead of $E$ to denote a particular exceptional object since in Section \ref{proofblpP2} we use $E$ for the exceptional divisor of $Bl_p\Ptoo$.


An \textbf{exceptional object} $E_\star \in D(S)$ is one with $\Homx{0}{E_\star}{E_\star}=\bC$ and $\Homx{i}{E_\star}{E_\star}=0$ for all $i\neq 0$. For a Del Pezzo surface $S$, Kuleshov and Orlov show that any exceptional object $E$ is either a Mumford $-K_S$-stable locally-free sheaf or a torsion sheaf of the form $\cO_C(d)$ with $C \subset S$ an irreducible rational curve satisfying $C^2=-1$ and $d\in\bZ$ an integer (see, e.g. \cite{helixtheory} or \cite[Theorem 8.1]{helicesondelpezzos}).
Note that the Del Pezzo condition is not necessary for much of the following discussion (we will make a note where it is used).

An \textbf{exceptional collection} $\bE \subset D(S)$ is a sequence $(E_1,\ldots,E_n)$ such that each $E_i$ is an exceptional object and  $i<j$ implies $\Homdot{E_j}{E_i}=0$. An exceptional collection $\bE$ is called \textbf{full} if  the smallest full triangulated subcategory of $D(S)$ containing $\bE$ is $D(S)$ itself, and $\bE$ is called \textbf{strong} if for $i<j$ we have $\Homx{k}{E_i}{E_j}=0$ unless $k=0$.

Bridgeland and Stern show that a full strong exceptional collection yields a heart $\cA_\bE \in D(S)$ (of a bounded $t$-structure) that is equivalent to the module category of a quiver algebra, which in turn is equivalent to the category of finite-dimensional representations of a quiver (possibly with relations) \cite[Theorem 2.4]{helicesondelpezzos}. Furthermore, the heart $\cA_\bE$ is the smallest full extension-closed subcategory of $D(S)$ containing the \textbf{dual collection}  $\bF=(F_n,\ldots,F_1)$ to $\bE$ (for this reason we often denote $\cA_\bE$ by $\cA_\bF$). The objects $F_i$ are defined by $F_i = L_{E_1} L_{E_2} \cdots L_{E_{i-1}}(E_i)$, where $L_A B$ is the \textbf{left mutation} of $B$ through $A$ defined by the canonical evaluation  triangle:
$$ \Homdot{A}{B}\otimes A \to B \to L_A B.$$
The quiver associated to the heart $\cA_\bE$ has a vertex associated to each $E_i$. The number of arrows $n_{ij}$ from vertex $i$ to vertex $j$ can be obtained using either irreducible hom's between objects in $\bE$ or extensions of objects in $\bF$:
\begin{itemize}
\item $n_{ij} = \dim \text{ coker } \displaystyle\bigoplus_{i<k<j} \Homx{}{E_i}{E_k}\otimes\Homx{}{E_k}{E_j} \to \Homx{}{E_i}{E_j}$
\item $n_{ij} = \dim \Homx{1}{F_j}{F_i}$
\end{itemize}
The objects $F_i$ correspond to the simple representation over the $i^{\text{th}}$ vertex.


Given a full exceptional collection $\bE=(E_1,\ldots,E_n)$ one may generate the \textbf{helix} $\bH = (E_i)_{i\in\bZ}$ using the rule $E_{i-n} = E_i \otimes \omega_S$. (This generates a helix of type $(n,3)$ in the notation of \cite{helicesondelpezzos}.) The helix $\bH$ is said to be \textbf{geometric} if $i<j$ implies $\Homx{k}{E_i}{E_j}=0$ unless $k=0$. If $\bH$ is a geometric helix then each ``thread'' $(E_k,\ldots,E_{k+n})$ is a full strong exceptional collection.

Bridgeland and Stern define an operation on exceptional collections (resp. helices) called \textbf{mutation defined by a height function} for $E_\star$, which constructs a new exceptional collection $\bE'$ (resp. helix $\bH'$) from a given one $\bE$ (resp. $\bH$) and choice of object $E_\star \in \bE$ (resp. $\bH)$. If $\bH$ is geometric then so is $\bH'$ and it is shown that, if $S$ is a Del Pezzo surface, then for any choice of object $E_\star \in \bH$ there is an associated height function. We do not formally define this operation here, but direct the interested reader to Appendix \ref{tiltingquivers} for the precise definition as well as an interpretation using the quiver associated to a thread containing $E_\star$. 

Because the quiver algebra associated to $\bE'$ is the left tilt of that associated to $\bE$ (\cite[Proposition 7.3]{helicesondelpezzos}) and since for Del Pezzos the mutation by a height function operation is determined solely by the choice of object $E$ (see Proposition \ref{quivermutprop}), we refer to a mutation defined by a height function for $E$ as a \textbf{left tilt} at $E$.



\section{Discussion of the strategy}\label{strategy}

Let us now discuss the specifics of our strategy.



From Section \ref{exccolhts}, a full, strong exceptional collection $\bE=(E_1,\ldots,E_n)$ with dual collection $\bF=(F_n,\ldots,F_1)$ yields a heart $\cA_\bF \subset D(S)$ that is generated by extensions of the objects $F_i$, and equivalent to finite-dimensional representations of a quiver. The exceptional collection $\bF$ is ``Ext'' in the sense of \cite[Definition 3.10]{macri} and so by \cite[Lemma 3.16]{macri}, if $\sigma=(Z,\cA)\in \stabdivS$ and $\bF\subset \cA$ then $\cA = \cA_\bF$.

A choice of invariants $v$ corresponds to a choice of dimension vector and following the proof of \cite[Proposition 8.1]{ABCH} we see that $\sigma$-stability is equivalent to King's $\chi$-stability \cite{kingquiver}. Thus the moduli space $\cM_\sigma(v)$ of Bridgeland semistable objects is projective when semistable objects are considered (and if only stable objects are considered, the space is quasi-projective).

We want to use the above observations to deduce that $\cM_\sigma(v)$ is projective for all $\sigma \in \stabdivS$ and (Bogomolov) $v$. The first issue with this strategy is that there is no $\sigma \in \stabdivS$ with $\bF\subset \cA$. This is because all dual collections we consider have objects $G[2]$ where $G$ is a sheaf, and these objects cannot belong to any $\cA$ by definition. However, there are stability conditions such that, after a gentle operation called ``rotation'' (which does not affect stability, but does affect what shift of certain objects belong in the heart), the rotated stability condition $\sigma[\phi]=(Z[\phi],\cA[\phi])$ has $\bF \subset \cA[\phi]$. 

A rotation is defined as follows:
Given $\sigma=(Z,\cA) \in\stabdivS$ and $\phi\in (0,1)$, a rotation by $\phi$ yields the Bridgeland stability condition $\sigma[\phi]=(Z[\phi],\cA[\phi])$ where $Z[\phi](G) = e^{-\phi \pi i}Z(G)$ and $\cA[\phi]$ is the subcategory closed under extensions generated by the objects
\begin{itemize}
\item $G$ such that $G\in\cA$ is $\sigma$-semistable and $\text{arg }Z(G)/\pi > \phi$,
\item $G[1]$ such that $G\in\cA$ is $\sigma$-semistable and $\text{arg }Z(G)/\pi \leq \phi$.
\end{itemize}

In particular, if $G\in\cA$ is a $\sigma$-semistable object with arg $Z(G)/\pi\leq\phi$, then $G$ is ``replaced'' by $G[1]$ in $\cA[\phi]$.

%

We emphasize that rotation does not affect stability: $G$ is $\sigma$-(semi)stable iff $G$ is $\sigma[\phi]$-(semi)stable (recall that $G$ is $\sigma$-(semi)stable iff $G[k]$ is $\sigma$-(semi)stable for all $k$).

In Sections \ref{proofP1xP1} and \ref{proofblpP2} we find regions $R_\bF$ associated to a dual collection $\bF$ (where after a rotation by $\phi$, we have $\cA[\phi]=\cA_\bF$). We call $R_\bF$ a \textbf{quiver region} and  any $\sigma\in R_\bF$ a \textbf{quiver stability condition}. We find $R_\bF$ by determining the stability conditions $\sigma$ such that each object (or shift of an object) in $\bF$ is $\sigma$-semistable and where the objects have the correct Bridgeland slopes relative to each other so that after rotating we have $\bF \subset \cA[\phi]$. 

After finding a quiver region $R_\bF$, any choice of line bundle $L$ yields the quiver region $R_{\bF\otimes L}$ associated to the exceptional collection $\bE\otimes L = (E_1 \otimes L,\ldots,E_n\otimes L)$. Our discussion above shows that for any quiver stability condition $\sigma$ and choice of invariants $v$, we have $\cM_\sigma(v)$ projective. To extend this projectivity to other stability conditions, we ``slide down the wall'' in the following sense. 

In \cite{MaciociaWalls}, Maciocia presents real 3-dimensional slices of $\stabdivS$ determined by a choice of an ample $\bR$-divisor $H$ and a divisor orthogonal to $H$. When $S$ has Picard rank 2, these slices are determined solely by $H$, and are defined as
$$ \cS_H = \{ \sigma_{D,tH} | D\ \bR \text{-divisor}, t >0 \}.$$
We parametrize $\cS_H$ with the coordinates $(x,y,t)$ where the $xy$-plane is a parametrization of $D$ in $NS(S)\otimes \bR$. 

To determine $\sigma$-stability of objects and relative Bridgeland slopes, we must consider \textbf{walls} 
$$\cW(G,G') = \{ \sigma \in \stabdivS | \beta(G) = \beta(G')\}.$$
Maciocia shows that there are disjoint vertical\footnote{By ``vertical'' we mean defined by an equation involving only $x$ and $y$.}
planes $\Pi_u \subset \cS_H$ depending on a real parameter $u$ such that $\bigcup_{u\in\bR} \Pi_u = \cS_H$, and such that for any object $G$ whose Chern characters satisfy the Bogomolov inequality, the walls $\cW(G,\_) \cap \Pi_u$ are nested semi-circles (or vertical lines) \cite[Proposition 2.6]{MaciociaWalls}. 

\begin{figure}[H]
\begin{center}
\includegraphics[scale=.3]{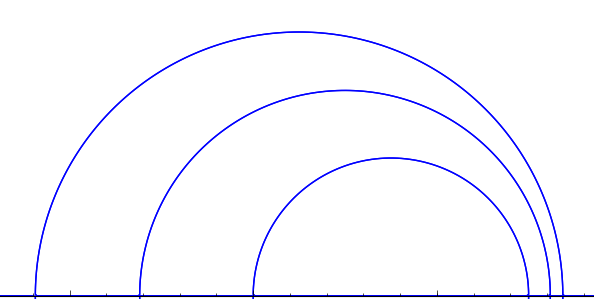}
\end{center}
\caption{Nested walls in $\Pi_u \subset \cS_H$}
\label{fig:nestedwalls}
\end{figure}

Any $\sigma \in \Pi_u$ lies on at most one of these nested semi-circles, say $\cW = \cW(G,G')\cap \Pi_u$. Since moving along $\cW$ crosses no other walls $\cW(G,\_)$, for $\sigma, \sigma' \in \cW$ we have $G$ $\sigma$-(semi)stable iff $G$ $\sigma'$-(semi)stable. More generally, since the walls $\cW(G,\_)=\cW(v,\_)$ are determined by the invariants of the objects involved, for $\sigma, \sigma' \in \cW$ we have $\cM_\sigma(v) = \cM_{\sigma'}(v)$. 

Thus, if $\cW \subset \Pi_u$ intersects a quiver region, then $\cM_\sigma(v)=\cM_{\sigma'}(v)$ for some quiver stability condition $\sigma'$ and  hence $\cM_\sigma(v)$ is projective. For each of $S= \PP$ and $S= Bl_p\Ptoo$ we find a quiver region $R_\bF$ such that the quiver regions $\bR_{\bF\otimes L}$ cover the entire $xy$-plane for each $\cS_H$. Since each $\sigma\in\cS_H$ is in a $\Pi_u$ plane, we  obtain the projectivity of $\cM_\sigma(v)$ for all Bogomolov $v$ and $\sigma\in\stabdivS$ and thus prove Conjecture \ref{con:main} in these cases.

We also use the fact that walls $\cW(v,\_)$ in $\Pi_u$ are nested semi-circles for simplification: To understand the geometry of a wall in $\cS_H$ it suffices to consider its intersection with the $xy$-plane.


\section{Proof of the conjecture for $S=\bP^1\times\bP^1$}\label{proofP1xP1}

In this section, we determine a suitable\footnote{Suitable in the sense that the entire $xy$-plane of $\cS_H$ is covered by tensoring with line bundles.} quiver region in $\stabdivS$ for $S=\bP^1\times\bP^1$ and use the region to conclude projectivity for the Bridgeland moduli spaces $\cM_\sigma(v)$ where the invariants $v$ satisfy the Bogomolov inequality (following the discussion of Section \ref{strategy}).

We shall denote by $\cO_S(a,b)$ the line bundle $p_1^*\cO_{\Proj}(a) \otimes p_2^*\cO_{\Proj}(b)$ with $p_1, p_2$ the natural projections from $\PP$ to the two copies of $\bP^1$, respectively. Also, we shall denote by $D_1$ and $D_2$ the divisors corresponding to $\cO_S(1,0)$ and $\cO_S(0,1)$, respectively. These are the generators of the cone of effective curves, and every other divisor can be written as $xD_1+yD_2$ for some $x,y \in\bR$. Note that $D_1^2=D_2^2=0$, and $D_1.D_2=1$.

A Bridgeland stability condition $\sigma_{D,H}\in\stabdivS$ is determined by real divisor classes $D=xD_1+yD_2$ and $H=aD_1+bD_2$, with $H$ ample. By the Nakai-Moishezon criterion, $H$ is ample iff $a,b>0$.

For a fixed $D,H$ pair, the corresponding Bridgeland stability condition $\sigma_{D,H}=\sigma=(\cA,Z)$ (we now drop the subscript $D,H$) has heart $\cA$ generated by
$\mu_H$-stable sheaves $G$ of slope $\mu_H(G)>bx+ay$,
torsion sheaves, and
objects of the from $G[1]$, where $G$ is a $\mu_H$-stable sheaf of slope $\mu_H(G)\leq bx+ay$.

The central charge $Z$ is defined as
$$ Z(G) = -\ch_2(G) + d_2x + d_1y + r (ab - xy) + i \left( d_2a + d_1b - r (bx + ay) \right), $$
where $\rk(G)=r$, and $c_1(G)=d_1D_1+d_2D_2$.

The slices $\cS_H\subset \stabdivS$ defined in Section \ref{strategy} are given here by  $\cS_H = \{ \sigma_{D,tH} | D = xD_1 + yD_2 \text{ and } t>0\}$. We identify the Bridgeland stability condition $\sigma=\sigma_{D,tH}$ where $D = xD_1 + yD_2$ with the coordinates $(x,y,t)$. Often we need only consider the $x$ and $y$ coordinates, and we project to the the $xy$-plane.


\subsection{A suitable exceptional collection}

In the following, we rely on the definitions given in Section \ref{exccolhts} and the specifics of the ``left tilt'' operation given in Appendix \ref{tiltingquivers}.

The natural exceptional collection
$$ \bE = ( \cO_S,\ \cO_S(1,0),\ \cO_S(0,1),\ \cO_S(1,1) ) $$
on $\PP$ is full and strong (and in fact generates a geometric helix), but does not have a large enough quiver region for our purposes,
so we find another exceptional collection.
The quiver associated to $\bE$ is

\begin{center}
\begin{tikzcd}
\overset{\cO_S}{\bullet} \arrow[yshift=0.7ex]{r} \arrow[yshift=-0.7ex]{r} \arrow[xshift=0.7ex]{d} \arrow[xshift=-0.7ex]{d}
& \overset{\cO_S(1,0)}{\bullet} \arrow[xshift=0.7ex]{d} \arrow[xshift=-0.7ex]{d} \\
\overset{\cO_S(0,1)}{\bullet} \arrow[yshift=0.7ex]{r} \arrow[yshift=-0.7ex]{r} 
& \overset{\cO_S(1,1)}{\bullet}
\end{tikzcd}
\end{center}
and we use this information to left tilt\footnote{For details on performing the left tilt operation, see Proposition \ref{quivermutprop}.} $\bE$ at $\cO_S(1,1)$ and obtain the exceptional collection
$$ \bE' = ( \cO_S,\ L_{\cO_S(1,0)}L_{\cO_S(0,1)}(\cO_S(1,1))[-1],\ \cO_S(1,0),\ \cO_S(0,1) ). $$

Straightforward calculations yield\footnote{See Appendix \ref{E'} for the calculations.}
$$ \bE' = ( E_1, E_2, E_3, E_4 ) = ( \cO_S,\ G,\ \cO_S(1,0),\ \cO_S(0,1) ), $$
where $G$ is a sheaf of rank 3, $c_1(G)=c_1(\cO_S(1,1))$, and $\ch_2(G)=-1$. The dual collection to $\bE'$ is is\footnote{See Appendix \ref{F'} for the calculations.}
$$ \bF' = ( F_4, F_3, F_2, F_1 ) = ( \cO_S(-2,-1)[2],\ \cO_S(-1,-2)[2],\ \cO_S(-1,-1)[1],\ \cO_S ). $$

Note that $\bF'$ is indeed an ``Ext'' exceptional collection in the sense of \cite[Definition 3.10]{macri}.
The heart $\cA_{\bF'}$ is naturally equivalent to finite-dimensional (contravariant) representations of the quiver

\begin{center}
\begin{tikzcd}
\overset{\cO_S}{\bullet} \arrow[yshift=0.7ex]{dr} \arrow[yshift=-0.7ex]{dr} \arrow[yshift=-1.4ex]{dr} \arrow{dr} 
 \\
& \overset{G}{\bullet} \arrow[yshift=0.7ex]{r} \arrow[yshift=-0.7ex]{r}  \arrow[xshift=-0.7ex]{d}  \arrow[xshift=0.7ex]{d} 
& \overset{\cO_S(1,0)}{\bullet}\\
& \overset{\cO_S(0,1)}{\bullet}
\end{tikzcd}
\end{center}

This quiver remains constant for any tensor of $\bE'$ and $\bF'$ by line bundles (with the labels above the vertices adjusted appropriately). We now locate the quiver region associated to $\bF'$.



\subsection{The associated quiver region}

It follows from \cite[Theorem 1.1]{linebundlessurfs} that all objects appearing in $\bF'$ (and their shifts) are $\sigma$-stable for all $\sigma\in\stabdivS$. Therefore, to find the associated quiver region, we need only find the stability conditions that can be rotated so the new heart contains $\cO_S(-2,-1)[2]$, $\cO_S(- 1,-2)[2]$, $\cO_S(-1,-1)[1]$, and $\cO_S$. We prove the following, where we restrict to a slice $\cS_H\subset \stabdivS.$

\begin{lem}
The quiver region $R_{\bF'} \subset \cS_H \subset \stabdivS$ associated to $$\bF' = ( \cO_S(-2,-1)[2],\ \cO_S(-1,-2)[2],\ \cO_S(-1,-1)[1],\ \cO_S )$$ is the region strictly inside both of the ellipsoidal walls $\cW(\cO_S,\cO_S(-2,-1)[1])$ and $\cW(\cO_S,\cO_S(-1,-2)[1])$.
\end{lem}

\begin{proof}

By the definition of the hearts $\cA$ associated to $\sigma=(Z,\cA)\in \stabdivS$, if $L$ is a line bundle on $S$ then either $L$ or $L[1]$ is in $\cA$. Also, if $G\in\cA$ is  $\sigma$-stable, then after rotating $\sigma$ with $0 < \phi< 1$, we have either $G$ or $G[1]$ in $\cA[\phi]$. Thus, to rotate $\sigma$ to $\sigma[\phi]$ and have $\cO_S,\ \cO_S(-2,-1)[2],\ \cO_S(-1,-2)[2] \in \cA[\phi]$ we must have 
\begin{equation}
\label{eqn:htp1p1}
\cO_S,\ \cO_S(-2,-1)[1],\ \cO_S(-1,-2)[1]\in\cA.
\end{equation}

To ensure that rotating does not force $\cO_S[1]\in\cA_\phi$ we must also have
\begin{equation}
\label{eqn:OandBottom}
\beta(\cO_S)>\beta(\cO_S(-2,-1)[1]),\ \beta(\cO_S(-1,-2)[1]).
\end{equation}

For any such $\sigma$, either $\cO_S(-1,-1)$ or $\cO_S(-1,-1)[1]$ is in $\cA$. If $\cO_S(-1,-1)\in\cA$, then we must rotate so that $\cO_S(-1,-1)[1]\in\cA[\phi]$. But then to ensure (as above) that rotating does not force $\cO_S[1]\in\cA[\phi]$ we must have
\begin{equation}
\label{eqn:Oand-1-1}
\beta(\cO_S)>\beta(\cO_S(-1,-1))
\end{equation}
Similarly, if $\cO_S(-1,-1)[1]\in\cA$, then we must rotate so that $\cO_S(-1,-1)[1]\in\cA_\phi$. To ensure that rotating does not force $\cO_S(-1,-1)[2]\in\cA[\phi]$ we must have
\begin{equation}
\label{eqn:-1-1andBottom}
\beta(\cO_S(-1,-1)[1])>\beta(\cO_S(-2,-1)[1]),\ \beta(\cO_S(-1,-2)[1])
\end{equation}

If $\sigma$ satisfies conditions (\ref{eqn:htp1p1}), (\ref{eqn:OandBottom}), and (\ref{eqn:Oand-1-1}) or if $\sigma$ satisfies conditions (\ref{eqn:htp1p1}), (\ref{eqn:OandBottom}), and (\ref{eqn:-1-1andBottom}), then we may rotate $\sigma$ to $\sigma[\phi]$ so that $\bF'\subset\cA[\phi]$. Then \cite[Lemma 3.16]{macri} implies $\cA[\phi] = \cA_{\bF'}$ and so $\sigma$ is a quiver stability condition. 

We now restrict to a particular slice $\cS_H \subset \stabdivS$ and determine the region $R_{\bF'}\subset \cS_H$ consisting of the $\sigma$ that satisfy one of these two sets of conditions. 

From the definition of the hearts $\cA$, condition (\ref{eqn:htp1p1}) states that for $\sigma \leftrightarrow (x,y,t)$ we must have $(x,y)$  to the left of the line $ay+bx=0$ and on or to the right of the lines $a(y+1)+b(x+2)=0$ and $a(y+2)+b(x+1)=0$. Note that $a,b>0$ implies that these lines are both diagonal (in fact, negatively sloped), so that the notions ``to the left'' and ``to the right'' are sensible.
(For the pictures below, we chose $a=2$ and $b=1$.)

\begin{figure}[htb]
\centering
\includegraphics[scale=.4]{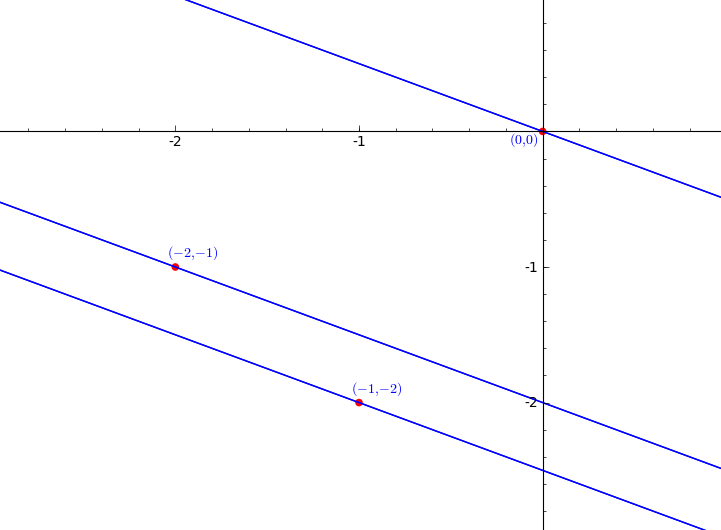}
\caption{Lines in $xy$-plane which determine the inclusion of $\cO_S,\ \cO_S(-2,-1)$ and $\cO_S(-1,-2)$ in $\cA$.}
\label{fig:slopelinesp1p1}
\end{figure}

To understand condition (\ref{eqn:OandBottom}), let us look at the walls $\cW(\cO_S,\cO_S(-2,-1)[1])$ and $\cW(\cO_S,\cO_S(-1,-2)[1])$.
We are working in $\cS_H$, which means that we have fixed an ample $H=aD_1+bD_2$, and we are considering stability conditions $\sigma$ depending on $D=xD_1+yD_2$ and $tH=taD_1+tbD_2$.
We have that
$$ Z(\cO_S) = (t^2ab - xy) - it(bx+ay), \textrm{ and } $$
$$ Z(\cO_S(-2,-1)[1]) = 2 + x + 2y - (t^2ab - xy) + it(a+2b+bx+ay). $$
The equation of the wall $\cW(\cO_S,\cO_S(-2,-1)[1])$, which is given by $\beta(\cO_S)=\beta(\cO_S(-2,-1)[1])$, simplifies to
$$ t^2ab(a+2b) + 2a(y^2+y) + b(x^2+2x) = 0, $$
which is an ellipsoid in the space $\cS_H$ parametrized by $x,y,t$.
The region where $\beta(\cO_S)>\beta(\cO_S(-2,-1)[1])$ is the region inside the ellipsoid.
A similar calculation gives us the following equation for $\cW(\cO_S,\cO_S(-1,-2)[1])$:
$$ t^2ab(2a+b) + a(y^2+2y) + 2b(x^2+x) = 0. $$

The intersections of the ellipsoidal walls with the $xy$-plane are the ellipses $2a(y^2+y)+b(x^2+2x)=0$ and $a(y^2+2y)+2b(x^2+x)=0$, respectively.
The line $ay+bx=0$ is tangent to both ellipses at $(0,0)$, and the line $a(y+1)+b(x+2)=0$ [resp.\ $a(y+2)+b(x+1)=0$] is tangent to the ellipse of $\cW(\cO_S,\cO_S(-2,-1)[1])$ [resp.\ $\cW(\cO_S,\cO_S(-1,-2)[1])$] at $(-2,-1)$ [resp.\ $(-1,-2)$]. Note that the vertical planes in $\cS_H$ over these lines do not intersect the region inside the two ellipsoidal walls, so any $(x,y,t)$ in the region inside both walls automatically satisfies condition (\ref{eqn:htp1p1}) (see Figure \ref{fig:twoellipsesp1p1}).

\begin{figure}[htb]
\centering
\includegraphics[scale=.4]{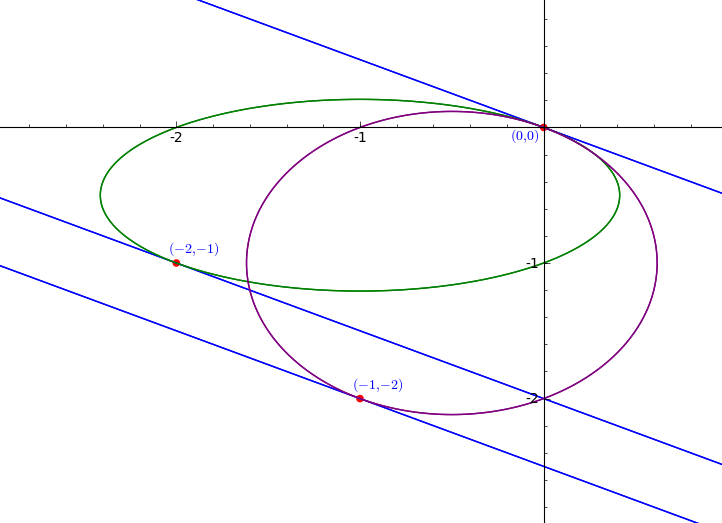}
\caption{Restriction of the walls $\cW(\cO_S,\cO_S(-2,-1)[1])$ and $\cW(\cO_S,\cO_S(-1,-2)[1])$ to the $xy$-plane.}
\label{fig:twoellipsesp1p1}
\end{figure}

We now show that conditions (\ref{eqn:Oand-1-1}) and (\ref{eqn:-1-1andBottom}) add no other restrictions. If $\cO_S(-1,-1)\in\cA$ then $\cO_S\in\cA$, and since $\cO_S(-1,-1)\su\cO_S$ and $\cO_S$ is $\sigma$-stable, we must have that $\beta(\cO_S)>\beta(\cO_S(-1,-1))$, as needed in condition (\ref{eqn:Oand-1-1}).
Similarly, if $\cO_S(-1,-1)[1]\in\cA$, then $\cO_S(-2,-1)[1],\ \cO_S(-1,-2)[1]\in\cA$. Moreover, since $\cO_S(-1,-1)[1]$ is $\sigma$-stable, and $\cO_S(-2,-1)[1]$ and has $\cO_S(-1,-2)[1]$ as subobjects, we must have $\beta(\cO_S(-1,-1)[1])>\beta(\cO_S(-2,-1)[1])$ and $\beta(\cO_S(-1,-1)[1])> \beta(\cO_S(-1,-2)[1])$ as needed in condition (\ref{eqn:-1-1andBottom}). 

Therefore the quiver region $R_{\bF'}\subset\cS_H$ is the intersection of the ellipsoidal regions bounded by the walls $\cW(\cO_S,\cO_S(-2,-1)[1])$ and $\cW(\cO_S,\cO_S(-1,-2)[1])$, as we set out to show.
\end{proof}

The projection of the quiver region $R_{\bF'}$ onto the $xy$-plane is a region containing a unit square with three corners cut off: $U=[-1,0]\times[-1,0]\smallsetminus\{(0,0), (-1,0), (0,-1)\}$. The analogous region $U(p,q)$ associated to the quiver region $R_{\bF'\otimes\cO_S(p,q)}$ is $U+(p,q)$, where the $+$ indicates component-wise addition, and together the regions $U(p,q)$ cover the entire $xy$-plane. Recalling the argument given in Section \ref{strategy}, we have shown the following.

\begin{pro}
\label{pro:ppconj}
Conjecture \ref{con:main} holds for $S=\PP$. In particular, if the invariants $v$ satisfy the Bogomolov inequality, then every moduli space $\cM_\sigma(v)$ is isomorphic to a moduli space $\cM_{\sigma'}(v)$ where $\sigma'$ is in a quiver region $R_{\bF'\otimes\cO_S(p,q)}$, and hence $\cM_\sigma(v)$ is projective.
\end{pro}

%
%
%
%


\section{Proof of the conjecture for $S=Bl_p\bP^2$}\label{proofblpP2}

In this section, we determine a suitable quiver region in $\stabdivS$ for $S=Bl_p\bP^2$. The considerations are similar to those for $S=\PP$, but with two exceptions: First, instead of a single quiver region, we find two quiver regions that together cover a ``unit region,'' and second, since the relevant hearts contain a torsion sheaf as one of the generators, the stability of those sheaves must be understood.

Let $H$ be the strict transform of the hyperplane class in $\bP^2$, let $E$ be the exceptional divisor, and let $F=H-E$.
We then have $H^2=1$, $E^2=-1$, $H.E=0$, $F^2=0$, and $E.F=1$.
The cone of effective curves on $S$ is the cone of non-negative linear combinations of $E$ and $F$.

A Bridgeland stability condition $\sigma_{D,H} \in \stabdivS$ is determined by real divisor classes $D=xE+yF$ and $H=aE+bF$, with $H$ ample. By the Nakai-Moishezon criterion, $H$ is ample iff $b>a>0$.

For a fixed $D,H$ pair, the corresponding Bridgeland stability condition $\sigma_{D,H}=\sigma=(Z,\cA)$ has heart $\cA$ generated by
$\mu_H$-stable sheaves $G$ of slope $\mu_H(G)>bx+ay-ax$,
torsion sheaves, and
objects of the from $G[1]$, where $G$ is a $\mu_H$-stable sheaf of slope $\mu_H(G)\leq bx+ay-ax$.

The central charge $Z$ is defined as
$$ Z(G) = -\ch_2(G) + d_Fx + d_Ey - d_Ex + r (ab - xy) + \frac{r}{2}(x^2 - a^2) $$
$$ + i \left( d_Fa + d_Eb - d_Ea - r (bx + ay - ax) \right). $$
where $\rk(G)=r$, and $c_1(G)=d_EE+d_FF$.

The slices $\cS_H\subset \stabdivS$ defined in Section \ref{strategy} are given here by  $\cS_H = \{ \sigma_{D,tH} | D = xE + yF \text{ and } t>0\}$. We identify the Bridgeland stability condition $\sigma=\sigma_{D,tH}$ where $D = xE + yF$ with the coordinates $(x,y,t)$. As before, we often project to the the $xy$-plane and only consider the $x$ and $y$ coordinates.


\subsection{Finding two suitable exceptional collections}

In \cite{helicesondelpezzos}, Bridgeland and Stern give an exceptional collection
$$ \bE = ( \cO_S,\ \cO_S(H-E),\ \cO_S(H),\ \cO_S(2H-E) ), $$ which we rewrite as
$$ \bE = ( \cO_S,\ \cO_S(F),\ \cO_S(E+F),\ \cO_S(E+2F) ). $$

The quiver for $\bE$ is

\begin{center}
\begin{tikzcd}
\overset{\cO_S}{\bullet} \arrow[yshift=0.7ex]{r} \arrow[yshift=-0.7ex]{r} \arrow{dr}
& \overset{\cO_S(F)}{\bullet} \arrow{d} \arrow{dr} \\
& \overset{\cO_S(E+F)}{\bullet} \arrow[yshift=0.7ex]{r} \arrow[yshift=-0.7ex]{r} 
& \overset{\cO_S(E+2F)}{\bullet}
\end{tikzcd}
\end{center}


and the dual collection $\bF$ to $\bE$ is
$$ \bF = ( \cO_S(-E-F)[2],\ \cO_S(-E)[1],\ \cO_S(-F)[1],\ \cO_S ). $$

The quiver region associated to this exceptional collection turns out to be too small for our purposes, i.e., it does not cover a ``unit region'' in the $xy$-plane of $\cS_H$. To cover a unit region, we combine the quiver regions from two  exceptional collections.








The first of these exceptional collections is found replacing $\bE$ (which generates a geometric helix) with its left tilt at $\cO_S(E+2F)$:
$$ \bE' = ( \cO_S,\ L_{\cO_S(F)}L_{\cO_S(E+F)}(\cO_S(E+2F))[-1],\ \cO_S(F),\ \cO_S(E+F) ). $$

Straightforward calculations show that\footnote{See Appendix \ref{E'2} for the calculations.}
$$ \bE' = ( \cO_S,\ G_1,\ \cO_S(F),\ \cO_S(E+F) ), $$
where $G_1$ is a sheaf of rank 2 with $c_1(G_1)=E+F$ and $\ch_2(G_1)=-1/2$. The dual collection $\bF'$ to $\bE'$ is\footnote{See Appendix \ref{F'2} for the calculations.}
$$ \bF' = ( \cO_S(-E-2F)[2],\ \cO_S(E)|_E[1],\ \cO_S(-E-F)[1],\ \cO_S ). $$



The second exceptional collection is obtained from $\bE$ in three steps. We first move $\bE$ one position forward along the helix it generates to obtain the exceptional collection
$$ \hat{\bE} = ( \cO_S(F),\ \cO_S(E+F),\ \cO_S(E+2F),\ \cO_S(2E+3F) ). $$


We then twist $\hat{\bE}$ by $\cO_S(-F)$ to obtain the exceptional collection
$$ \hat{\bE}\otimes\cO_S(-F) = ( \cO_S,\ \cO_S(E),\ \cO_S(E+F),\ \cO_S(2E+2F) ) $$
which has the folowing quiver:
\begin{center}
\begin{tikzcd}
\overset{\cO}{\bullet} \arrow{r} \arrow{dr}
& \overset{\cO(E)}{\bullet} \arrow[xshift=-0.7ex]{d} \arrow[xshift=0.7ex]{d} \\
& \overset{\cO(E+F)}{\bullet} \arrow{r} \arrow[yshift=1ex]{r} \arrow[yshift=-1ex]{r} 
& \overset{\cO(2E+2F)}{\bullet}
\end{tikzcd}
\end{center}


Our (second) desired exceptional collection is the left tilt of $\hat{\bE}\otimes\cO_S(-F)$ at $\cO_S(2E+2F)$,
$$ \bE'' = ( \cO_S,\ \cO_S(E),\ L_{\cO_S(E+F)}(\cO_S(2E+2F))[-1],\ \cO_S(E+F) ), $$
which is equal to\footnote{See Appendix \ref{E''} for the calculations.}
$$ \bE'' = ( \cO_S,\ \cO_S(E),\ G_2,\ \cO_S(E+F) ), $$
where $G_2$ is a sheaf of rank 2 with $c_1(G_2)=E+F$ and $\ch_2(G_2)=-1/2$. The dual collection $\bF''$ to $\bE''$  is\footnote{See Appendix \ref{F''} for the calculations.}
$$ \bF'' = ( \cO_S(-E-2F)[2],\ \cO_S(-F)[1],\ \cO_S(E)|_E,\ \cO_S ). $$

Notice the dual collections $\bF'$ and $\bF''$ contain not only (shifts of) line bundles, but also the torsion sheaf $\cO_S(E)|_E$ (and its shift $\cO_S(E)|_E[1]$).


\subsection{Bridgeland stability of $\cO_S(E)|_E$}
\label{sec:stabofte}

In order to find the quiver regions for the exceptional collections $\bE'$ and $\bE''$, we need to understand the Bridgeland stability of $\TE$ (the stability of line bundles is already understood from \cite{linebundlessurfs}). We prove that $\TE$ is Bridgeland semistable for all $\sigma\in\stabdivS$. To do this, we first study subobjects of $\TE$ in $\cA$ and show that the corresponding walls in $\cS_H\subset\stabdivS$ are hyperboloids or cones. This tells us that our walls must cross certain vertical planes associated to subobjects or quotients, and using Bertram's Lemma we obtain a contradiction.



\subsubsection{Subobjects and walls for $\cO_S(E)|_E$}\label{TEWalls}

Let $G\su\cO_S(E)|_E\in\cA$.
Then there exists a short exact sequence
$$ 0 \lfun G \lfun \cO_S(E)|_E \lfun Q \lfun 0 $$
in $\cA$, where $Q$ is the quotient.
The corresponding long exact sequence of cohomologies tells us that $H^{-1}(G)=0$, and therefore, $G=H^0(G)$ is a sheaf.
In particular, we have the long exact sequence of coherent sheaves
$$ 0 \lfun H^{-1}(Q) \lfun G \lfun \TE \lfun H^0(Q) \lfun 0. $$

We summarize a few useful consequences:

\begin{itemize}
\item
$H^0(Q)$ is either 0 or a torsion sheaf supported in dimension $0$.
Otherwise, $G$ and $H^{-1}(Q)$ would have the same rank and $c_1$ (and thus the same Mumford $H$-slope), making it impossible for $G$ to be in $\cA$ and $H^{-1}(Q)[1]$ to also be in $\cA$.
\item
In particular, if $\rk(G)=0$, then $G$ is a torsion sheaf supported on $E$.
\item
If $\rk(G)>0$, then $G$ must be torsion-free.
Indeed, since $H^{-1}(Q)$ is torsion-free, any torsion that $G$ had would map to $\TE$, and tors$(G)$ would have to be supported on $E$.
Then, $G/$tors$(G)$ and $H^{-1}(Q)$ would have the same $c_1$ and rank, making it impossible for $G\in\cA$ and $H^{-1}(Q)[1]\in\cA$.
\end{itemize}

In a given slice $\cS_H\subset\stabdivS$ for a fixed ample divisor $H=aE+bF$, we are working with stability conditions $\sigma$ depending on $D=xE+yF$ and $tH=taE+tbF$.
We have that
$$ Z(\cO_S(E)|_E) = \frac12 + y - x + i t (b - a), $$
and the equation for the wall $\cW(G,\TE)$ simplifies into
$$ t^2r(a(b-a)(2b-a)) + (a-b)rx^2 - 2arxy + 2ary^2 + $$
$$ + (2bd_F+(b-a)r)x + (r-2d_F)ay + 2ac + ad_E - ad_F - 2bc - bd_E = 0, $$
where $\rk(G)=r$, $c_1(G)=d_EE+d_FF$, and $\ch_2(G)=c$.

If we restrict to any horizontal plane of equation $t=\textrm{constant}$, we obtain
a conic with discriminant
$$ (-2ar)^2 - 4 ((a-b)r) (2ar) = 4a^2r^2 - 8a^2r^2 + 8abr^2 = 4ar^2(2b-a)>0. $$
Therefore, these conics 
are all hyperbolas or cones, and the walls $\cW(G,\TE)$ in $\cS_H$ must be hyperboloids or cones.


\subsubsection{Betram's Lemma}


A key tool in our proof of the stability of $\TE$ is Bertram's Lemma.
We recall here what we need for our case. For more details, we point the reader to \cite[Lemma 6.3]{ABCH} and \cite[Lemma 4.7]{linebundlessurfs}.

Let $G$ be a subobject of $\TE$ of rank $\rk(G)>0$ in $\cA$, and let $Q$ be the quotient. Let $0=G_0\su G_1\su G_2\su\cdots\su G_{n-1}\su G_n=G$ be the Harder-Narasihman filtration of $G$, and let $K_i=G_i/G_{i-1}$ (so that $K_1=G_1$ and $K_n=G/G_{n-1}$).
We have that $\mu_H(K_1)>\mu_H(K_2)>\cdots>\mu_H(K_n)$ and we set $\Kl:=K_n$. Similarly, let $0=G'_0\su G'_1\su G'_2\su\cdots\su G'_{m-1}\su G'_m=H^{-1}(Q)$ be the Harder-Narasihman filtration of $H^{-1}(Q)$, and let $J_i=G'_i/G'_{i-1}$ with $\Jl:=J_1$.

We have that $G\su\TE\in\cA$ iff $\mu_H(\Jl)\leq D.H/H^2<\mu_H(\Kl)$.
Since $D.H/H^2=(ay+(b-a)x)/(2ab-a^2)$, note that for a fixed $\mu\in\bR$ the equation $D.H/H^2 = \mu$ gives a linear equation of the form $(b-a)x+ay=$const, and so defines a vertical plane in $\cS_H$.
We are saying here that $G\su\TE\in\cA$ for a stability condition $\sigma$ if and only if the corresponding point $(x,y,t)$ in $\cS_H$ is between the vertical planes $\mu_H(\Jl)=D.H/H^2$ and $\mu_H(\Kl)=D.H/H^2$ (the point could be on the first plane, but not the second one).

At $\mu_H(\Kl)=D.H/H^2$, we will consider the natural subsheaf $G_{n-1}\su G$.
At $\mu_H(\Jl)=D.H/H^2$, we will consider the natural quotient sheaf $G\sur G/\Jl$ (note that, as sheaves, $\Jl=J_1=G'_1\su H^{-1}(Q)\su G$).

\begin{lem}[Bertram's Lemma]\label{BertramLemma}
Fix $D$ and $H$ as above, and let $G\su\TE$ in $\cA$ for some $\sigma=\sigma_{D,H}=(Z,\cA)$ such that $\sigma\in\cW(G,\TE)$.
\begin{itemize}
\item[$(1)$]
If, for some $u$, $\cW(G,\TE)\cap\Pi_u$ intersects the line $D.H/H^2=\mu_H(\Kl)$ for $t>0$, then $\beta(G_{n-1})>\beta(G)$ at $\sigma$, with $G_{n-1}\su\TE$ in $\cA$ $($in particular, $G$ is not $\mu_H$-semistable$)$.
\item[$(2)$]
If, for some $u$, $\cW(G,\TE)\cap\Pi_u$ intersects the line $D.H/H^2=\mu_H(\Jl)$ for $t>0$, then $\beta(G/\Jl)>\beta(G)$ at $\sigma$, with $G/\Jl\su\TE$ in $\cA$.
\end{itemize}
\end{lem}

Note: The $\Pi_u$ in Bertram's Lemma are the vertical walls that we mentioned in Section \ref{strategy}.

\pf
The proof is essentially identical to that of \cite[Lemma 4.7]{linebundlessurfs} and we omit the details.
\qed


\subsubsection{Proof of stability}

We are now ready to prove the stability of $\TE$.

\begin{thm}
\label{thm: testable}
The torsion sheaf $\TE$ is Bridgeland stable for all $\sigma\in\stabdivS$.
\end{thm}

\pf
We prove that $G\su\TE$ in $\cA$ cannot satisfy $\beta(G)\geq\beta(\TE)$ at $\sigma$ by induction on $\rk(G)$ (starting at $\rk(G)=0$). 

Let $\sigma=(Z,\cA)\in\stabdivS$ and let $G$ be a torsion sheaf of rank $0$ such that $G\su\TE$ in $\cA$.
As we saw above, we must have $c_1(G)=E$.
The quotient of $\TE$ by $G$ in $\cA$ must be a torsion sheaf $Q$ supported on a scheme $P$ of dimension $0$.
In particular, it would have $c_1(Q)=0$ and $\ch_2(Q)=l\geq0$ (with equality only if $G=\TE$), where $l$ is the length of $P$. In this case, $Z(Q)=-l$, and therefore $\beta(\TE) > \beta(G)$.

Assume now that the result is true for all subobjects of rank $\leq r-1$, and let $G\su\TE$ be a subobject of $\TE$ in $\cA$ with $\rk(G)=r>0$.

Assume by way of contradiction that there is a $\sigma\in\cS_H\subset\stabdivS$ where $G$ satisfies $\beta(G)\geq\beta(\TE)$.

In $\cS_H$, the wall $\cW(G,\TE)$ is a hyperboloid or a cone, with $\beta(G)\geq\beta(\TE)$ satisfied on or inside the wall.
Moreover, for $G\in\cA$, we must have that the $\sigma$ stability condition is between the vertical planes $D.H/H^2=\mu_H(\Kl)$ or $D.H/H^2=\mu_H(\Jl)$.
Therefore, the hyperboloid or cone wall is forced to cross at least one of those two vertical planes. In either case, since $\bigcup_{u\in\bR}\Pi_u=\cS_H$, there exists a $u$ such that $\Pi_u \cap \cW(G,\TE) \neq \varnothing$, and Bertram's Lemma gives us a subobject $\hat{G} (=G_{n-1}$ or $G/\Jl)$ of $\TE$ satisfying $\beta(\hat{G})>\beta(G)\geq\beta(\TE)$. Since $\rk(\hat{G})<\rk(G)$, this contradicts our induction hypothesis.
\qed


\subsection{Quiver region for $\bE'$}

We will now calculate the quiver region for the exceptional collection $$ \bE' = ( \cO_S,\ G_1,\ \cO_S(F),\ \cO_S(E+F) ).$$ For this, we must find the stability conditions $\sigma$ that can be rotated so that the objects in associated dual collection $\bF'$ are in the new heart and are $\sigma$-stable. Our considerations proceed similarly to those made for the case $S=\PP$, but with the exception that line bundles (and their shifts) are no longer always Bridgeland stable. In addition, the walls we consider are no longer all ellipsoids, but can be hyperboloids.

We prove the following, where we restrict to a slice $\cS_H \subset \stabdivS$.

\begin{lem}
The quiver region $R_{\bF'} \subset \cS_H \subset \stabdivS$ associated to $$ \bF' = ( \cO_S(-E-2F)[2],\ \cO_S(E)|_E[1],\ \cO_S(-E-F)[1],\ \cO_S ) $$ is the region striclty inside both of the walls $\cW(\cO_S, \TE)$, which is a hyperboloid, and $\cW(\cO_S,\cO_S(-E-2F)[1])$, which is an ellipsoid.
\end{lem}

\begin{proof}

By the definition of the hearts $\cA$ associated to $\sigma=(Z,\cA) \in \stabdivS$, we have $\TE\in\cA$ for all $\sigma$. For any line bundle $L$ on $S$, either $L$ or $L[1]$ is in $\cA$. If $G\in \cA$, then rotating $\sigma$ with $0< \phi < 1$, we have either $G$ or $G[1]$ in $\cA[\phi]$. To rotate $\sigma$ to $\sigma[\phi]$ and have $\cO_S,\ \cO_S(-E-2F)[2] \in \cA[\phi]$ we must have

\begin{equation}
\label{eqn:0-1-2inht}
\cO_S,\ \cO_S(-E-2F)[1]\in\cA.
\end{equation}

To ensure that rotating yeilds $\cO_S(-E-2F)[2]$ and $\TE[1]$ in $\cA[\phi]$ (but not $\cO_S[1]$) we must also have

\begin{equation}
\label{eqn:te-1-2relslope}
\beta(\cO_S) > \beta(\TE),\ \beta(\cO_S(-E-2F)[1]).
\end{equation}

For any such $\sigma$, either $\cO_S(-E-F)$ or $\cO_S(-E-F)[1]$ is in $\cA$. If $\cO_S(-E-F) \in \cA$ then we must rotate so that $\cO_S(-E-F)[1]\in\cA[\phi]$. But then to ensure (as above) that rotating does not force $\cO_S[1] \in\cA[\phi]$ we must have

\begin{equation}
\label{eqn:-1-1relslope}
\beta(\cO_S)>\beta(\cO_S(-E-F))
\end{equation}

Similarly, if $\cO_S(-E-F)[1]\in\cA$, then we must rotate in such a way to keep $\cO_S(-E-F)[1]\in\cA[\phi]$. Since $\TE$ and $\cO_S(-E-2F)[1]$ must be shifted in the rotation, this gives

\begin{equation}
\label{eqn:-1-1[1]relslope}
\beta(\cO_S(-E-F)[1])>\beta(\TE),\ \beta(\cO_S(-E-2F)[1]).
\end{equation}

If $\sigma$ satisfies conditions conditions (\ref{eqn:0-1-2inht}), (\ref{eqn:te-1-2relslope}) and (\ref{eqn:-1-1relslope}) or if $\sigma$ satisfies conditions (\ref{eqn:0-1-2inht}), (\ref{eqn:te-1-2relslope}) and (\ref{eqn:-1-1[1]relslope}), and in either case if it is also true that each object of $\bF'$ is $\sigma$-stable, then $\cA[\phi]=\cA_{\bF'}$ and so $\sigma$ is a quiver stability condition.

\begin{figure}[h]
\centering
\includegraphics[scale=.4]{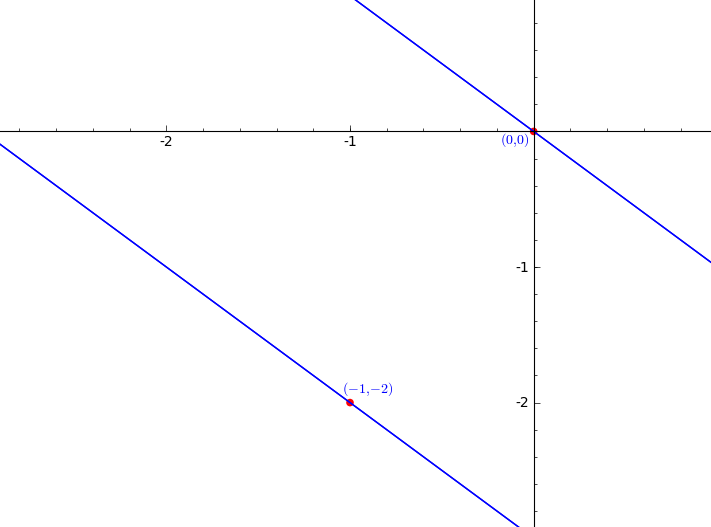}
\caption{Lines in the $xy$-plane that determine the inclusion of $\cO_S$ and $\cO(-E-2F)$ in $\cA$.}
\label{fig:slopelinesblp2}
\end{figure}

\textit{We now find the region in $\cS_H$ that the above conditions define.}

From the definition of the hearts $\cA$, condition (\ref{eqn:0-1-2inht}) states that for $\sigma \leftrightarrow (x,y,t)$ we must have $(x,y)$
to the left of the line $ay+(b-a)x=0$ and on or to the right of the line $a(y+2)+(b-a)(x+1)=0$. Note that since $b>a>0$, these lines are both negatively sloped and so the notions of ``to the left'' and ``to the right'' are sensible (see Figure \ref{fig:slopelinesblp2}).

To understand condition (\ref{eqn:te-1-2relslope}), we need to study the two walls $\cW(\cO_S,\TE)$ and $\cW(\cO_S,\cO_S(-E-2F)[1])$ in $\cS_H$.
In Section \ref{TEWalls}, we already saw the equation for a wall of the form $\cW(G,\TE)$, and so the equation for $\cW(\cO_S,\TE)$ is
$$ t^2 (a^3 - 3a^2b + 2ab^2) + (a-b)x^2 - 2axy + 2ay^2 + (b-a)x + ay = 0. $$
It is a hyperboloid (see Section \ref{TEWalls}), and its restriction to the $xy$-plane is a hyperbola passing through $(0,0)$ and $(-1,-2)$.
For the second wall, we have that
$$ Z(\cO_S)=t^2ab-xy+\frac12(x^2-t^2a^2)-it(bx+ay-ax), \textrm{ and } $$
$$ Z(\cO_S(-E-2F)[1])=\frac32+x+y-t^2ab+xy-\frac12(x^2-t^2a^2) $$
$$ \hspace{6.2cm} +it(bx+ay-ax+a+b). $$
Therefore, the wall $\cW(\cO_S,\cO_S(-E-2F)[1])$ has equation 
$$ t^2(a(a+b)(2b-a)) + (3b-a)x^2 - 2axy + 2ay^2 + 3(b-a)x + 3ay = 0. $$
If we restrict to any horizontal plane of equation $t=\textrm{constant}$, we obtain a conic with discriminant
$$ (-2a)^2 - 4 (3b-a) (2a) = 12a^2 - 24ab = 12a(a-2b) < 0. $$
Therefore, the wall $\cW(\cO_S,\cO_S(-E-2F)[1])$ is an ellipsoid.

In the $xy$-plane, the line $ay+(b-a)x=0$ is tangent to both walls at $(0,0)$, and the line $a(y+2)+(b-a)(x+1)=0$ is tangent to the ellipse $\cW(\cO_S,\cO_S(-E-2F)[1])$ (See Figure \ref{fig:ellhypblp2}). Note that the vertical planes in $\cS_H$ over these lines do not intersect the region inside both of the walls, so any $(x,y,t)$ in the region satisfying condition (\ref{eqn:te-1-2relslope}) automatically satisfies condition (\ref{eqn:0-1-2inht}).
Therefore, the region that satisfies conditions (\ref{eqn:0-1-2inht}) and (\ref{eqn:te-1-2relslope}) is the region strictly inside both of the walls $\cW(\cO_S,\TE)$ and $\cW(\cO_S,\cO_S(-E-2F)[1])$.

\begin{figure}[h]
\centering
\includegraphics[scale=.4]{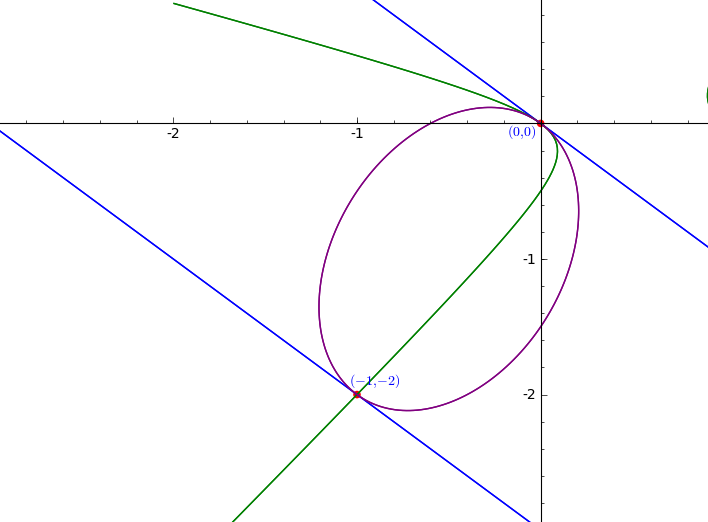}
\caption{Restriction of the walls $\cW(\cO_S,\TE)$ and $\cW(\cO_S,\cO_S(-E-2F)[1])$ to the $xy$-plane.}
\label{fig:ellhypblp2}
\end{figure}

We now show that conditions (\ref{eqn:-1-1relslope}) and (\ref{eqn:-1-1[1]relslope}) add no other restrictions.
For condition (\ref{eqn:-1-1relslope}), note that, if $\cO_S(-E-F) \in \cA$, then $\cO_S \in\cA$.
Since $(-E-F)^2>0$, \cite[Proposition 5.1, Remark 5.2]{linebundlessurfs} gives us that $\beta(\cO_S) > \beta(\cO_S(-E-F))$ as needed. 





Suppose, on the other hand, that $\cO_S(-E-F)[1]\in\cA$.
Then $\cO_S(-E-2F)[1] \su \cO_S(-E-F)[1] \in \cA$, and since $F^2=0$, we have that $\cO_S(-E-F)[1]$ does not destabilize $\cO_S(-E-2F)[1]$, which gives us $\beta(\cO_S(-E-F)[1]) > \beta(\cO_S(-E-2F)[1])$, as needed.

For the second inequality in Condition (\ref{eqn:-1-1[1]relslope}) we need to study the wall $\cW(\cO_S(-E-F)[1],\TE)$.
It has equation
$$ t^2(a(b-a)(2b-a)) + 
(a-b)x^2 - 2axy + 2ay^2 + (-a-b)x + 3ay + a = 0. $$
It is a hyperboloid and its restriction to the $xy$-plane goes through the points $(-1,-1)$ and $(0,-1/2)$.
The inequality $\beta(\cO_S(-E-F)[1])>\beta(\TE)$ of condition (\ref{eqn:-1-1[1]relslope}) is satisfied by all $(x,y,t)$ outside the hyperboloid.
Consider the hyperbola which is the restriction of the wall $\cW(\cO_S(-E-F)[1],\TE)$ to the $xy$-plane.
The tangent line at $(-1,-1)$ is the line $a(y+1)+(b-a)(x+1)=0$, and the left side of the hyperbola lies to the left of that line.
But this is the region where $\cO_S(-E-F)\in\cA$, which is not the case that we are considering.
The right side of the hyperbola touches the wall $\cW(\cO_S,\TE)$ at $(0,-1/2)$, but it is otherwise outside of it (See Figure \ref{fig:-E-FandTEblp2}).
Therefore, condition (\ref{eqn:-1-1[1]relslope}) is satisfied in the region where condition (\ref{eqn:te-1-2relslope}) is satisfied.


\begin{figure}[h]
\centering
\includegraphics[scale=.4]{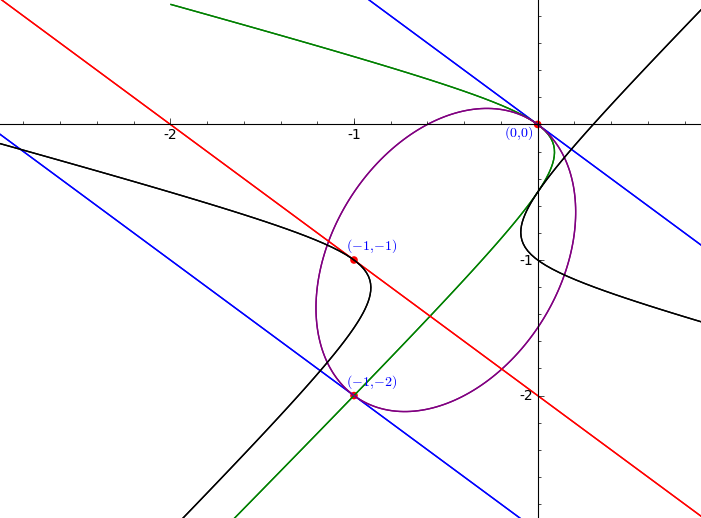}
\caption{The wall $\cW(\cO_S(-E-F)[1],\TE)$ does not share any region with $\cW(\cO_S,\TE)$ to the right of the line $a(y+1)+(b-a)(x+1)=0$.}
\label{fig:-E-FandTEblp2}
\end{figure}


In conclusion, the conditions (\ref{eqn:0-1-2inht}), (\ref{eqn:te-1-2relslope}) and (\ref{eqn:-1-1relslope}) (or (\ref{eqn:0-1-2inht}), (\ref{eqn:te-1-2relslope}) and (\ref{eqn:-1-1[1]relslope})) are satisfied in the region inside both the hyperboloidal wall $\cW(\cO_S,\TE)$ and the ellipsoidal wall $\cW(\cO(-E-2F)[1],\cO_S)$.
Let us call this region $R$. It is the region that we claim to be $R_{\bF'}$.

\textit{To prove that $R=R_{\bF'}$, we must show that all objects of $$\bF' = ( \cO_S(-E-2F)[2],\ \cO_S(E)|_E[1],\ \cO_S(-E-F)[1],\ \cO_S ) $$ are $\sigma$-stable for $\sigma \in R$. }

By this we mean that, for a given $\sigma\in R$, we need to prove that, whichever shift of $\cO_S(-E-2F)$, $\cO_S(E)|_E$, $\cO_S(-E-F)$, and $\cO_S$ is in $\cA$, it is $\sigma$-stable.

By Theorem \ref{thm: testable}, $\TE$ is $\sigma$-stable for all $\sigma\in\stabdivS$. For a line bundle $L$ on $S$,  \cite[Proposition 5.12, Lemma 6.2]{linebundlessurfs} implies that $L$ is $\sigma$-stable iff it is not destabilized by the inclusion $L(-E)$, and $L[1]$ is $\sigma$-stable iff it is not destabilized by the surjection $L[1] \sur L(E)[1]$.




It follows that $\cO_S$ is $\sigma$-stable when $\sigma$ is not in the region bounded by the wall $\cW(\cO_S(-E),\cO_S)$.
Since
$$ Z(\cO_S(-E)) = \frac12 + x - y + t^2ab - xy + \frac12 (x^2 - t^2a^2) + i (a - b - (bx+ay-ax)), $$
the equation of the wall is
$$ t^2(a(2b-a)(b-a)) + (a-b)x^2 - 2axy + 2ay^2 + (a-b)x - ay = 0. $$
It is a hyperboloid, and its restriction to the $xy$-plane passes through the point $(-1,-1/2)$ where it is tangent to the ellipsoidal wall $\cW(\cO_S(-E-2F)[1],\cO_S)$, but staying completely on the other side of the tangent line (See Figure \ref{fig:hypsnotinell}).
Thus $\cO_S$ is $\sigma$-stable for all $\sigma\in R$. 

Similarly, $\cO_S(-E-2F)[1]$ could only be destabilized by the quotient $\cO_S(-2F)[1]$.
But the wall $\cW(\cO(-E-2F)[1],\cO_S(-2F)[1])$ is actually the same as the wall $\cW(\cO_S(-E),\cO_S)$ moved down by 2 units in the $y$ direction.
Its restriction to the $xy$-plane passes through the point $(0,-3/2)$, where it is tangent to the ellipsoidal wall $\cW(\cO_S(-E-2F)[1],\cO_S)$, but staying completely on the other side of the tangent line (See Figure \ref{fig:hypsnotinell}).
Therefore, $\cO_S(-E-2F)[1]$ is also $\sigma$-stable for all $\sigma\in R$.

\begin{figure}[h]
\centering
\includegraphics[scale=.4]{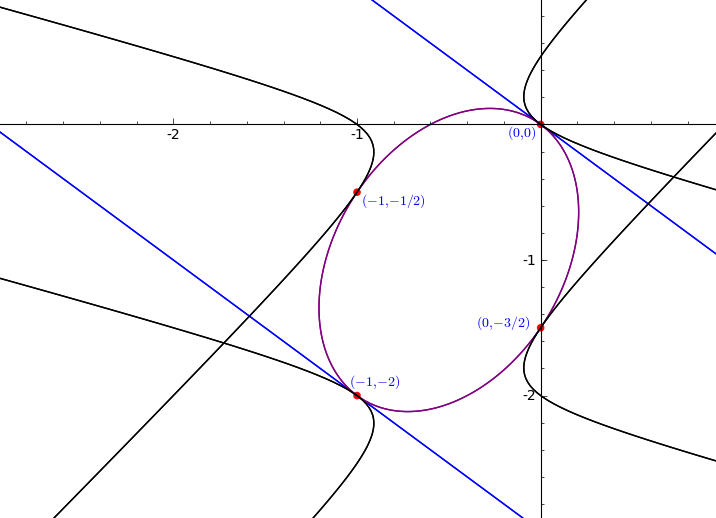}
\caption{The wall $\cW(\cO_S(-E-2F)[1],\cO_S)$ together with $\cW(\cO_S(-E),\cO_S)$ and $\cW(\cO_S(-E-2F)[1],\cO_S(-2F)[1])$.}
\label{fig:hypsnotinell}
\end{figure}

If $\cO_S(-E-F)\in\cA$, then it could only be destabilized by $\cO_S(-2E-F)$.
The wall $\cW(\cO_S(-2E-F),\cO_S(-E-F))$ can be seen to be the same as the wall $\cW(\cO_S(-E),\cO_S)$ moved one unit down and one unit to the left (see Figure \ref{fig:O2EFOE2F}).
The right side of the hyperboloidal wall $\cW(\cO_S(-2E-F),\cO_S(-E-F))$ is in the region where $\cO_S(-E-F)\not\in\cA$, so we do not have to worry about it in this case.
The left side of the wall is not in our region $R$.
Indeed, it can easily be checked that the two walls $\cW(\cO_S(-2E-F),\cO_S(-E-F))$ and $\cW(\cO_S(-E),\cO_S)$ do not cross the $x=-3/2$ vertical plane.
The left side of the hyperboloid $\cW(\cO_S(-2E-F),\cO_S(-E-F))$ is to the left of it, while the ellipsoid $\cW(\cO_S(-E),\cO_S)$ is to the right of it (see Figure \ref{fig:O2EFOE2F}).

\begin{figure}[h]
\centering
\includegraphics[scale=.4]{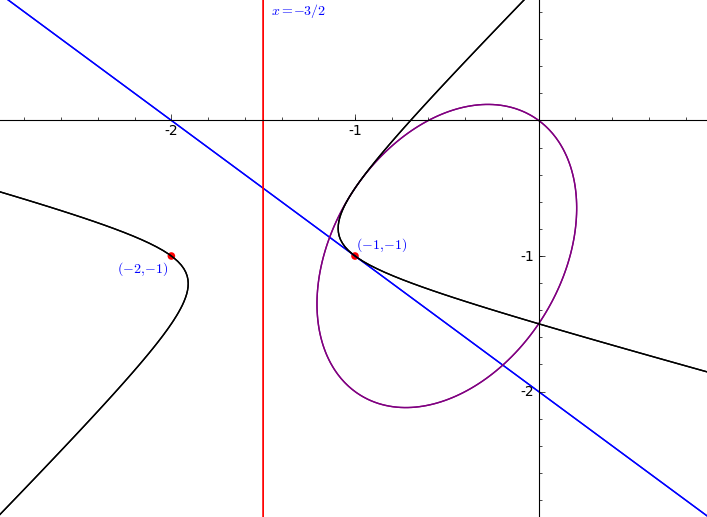}
\caption{The walls $\cW(\cO_S(-2E-F),\cO_S(-E-F))$ and $\cW(\cO_S(-E),\cO_S)$.}
\label{fig:O2EFOE2F}
\end{figure}

If, on the other hand, $\cO_S(-E-F)[1]\in\cA$, then it could only be destabilized by the quotient $\cO_S(-F)[1]$.
However, a quick calculation shows that the wall $\cW(\cO_S(-E-F)[1],\cO_S(-F)[1])$ is equal to the wall $\cW(\cO_S(-E-F)[1],\TE)$, and we already saw that such wall lies outside of our region $R$ when we look in the region where $\cO_S(-E-F)[1]\in\cA$ (see Figure \ref{fig:-E-FandTEblp2}).

Thus all objects of $\bF'$ are $\sigma$-stable for $\sigma \in R$ and we have shown that $R=R_{\bF'}$, the quiver region associated to $\bF'$.
\end{proof}



The translations of $R_{\bF'}$ by tensoring with line bundles do not cover the entire $xy$-plane. We now find the quiver region associated to $\bF''$ and show that the translations of both quiver regions together cover the $xy$-plane.





\subsection{Quiver region for $\bE''$ and the combined regions}

Determining the quiver region $R_{\bF''}$ associated to $$ \bE'' = ( \cO_S,\ \cO_S(E),\ G_2,\ \cO_S(E+F) ) $$ and its dual collection $$ \bF'' = ( \cO_S(-E-2F)[2],\ \cO_S(-F)[1],\ \TE,\ \cO_S ) $$
is very similar to the determination of the region $\bR_{\bF'}$.
Therefore, we summarize the argument and omit the details:

To have $\sigma \in R_{\bF''}$ we must have
\begin{itemize}
\item
$\cO_S, \TE, \cO_S(-E-2F)[1]\in\cA$
\item
$\beta(\cO_S), \beta(\TE)>\beta(\cO_S(-E-2F)[1])$.
\end{itemize}

Moreover, one of the following two things must happen:
\begin{itemize}
\item
$\cO_S(-F)\in\cA$, and $\beta(\cO_S), \beta(\TE) > \beta(\cO_S(-F))$, or
\item
$\cO_S(-F)[1]\in\cA$, and $\beta(\cO_S(-F)[1])>\beta(\cO_S(-E-2F)[1])$.
\end{itemize}
Also, all objects of $\bF''$ must also be $\sigma$-stable.

Any $\sigma$ satisfying the above is a quiver stability condition, and we find that $\bR_{\bF''}$ is the region strictly inside both the ellipsoidal wall $\cW(\cO_S,\cO_S(-E-2F)[1])$ and the hyperboloidal wall $\cW(\TE,\cO_S(-E-2F)[1])$. Figure \ref{fig:quiverregion''} shows the intersection of these walls with the $xy$-plane, as well as the lines that determine if $\cO_S, \cO_S(-E-2F)[1]\in\cA$.

\begin{figure}[h]
\centering
\includegraphics[scale=.4]{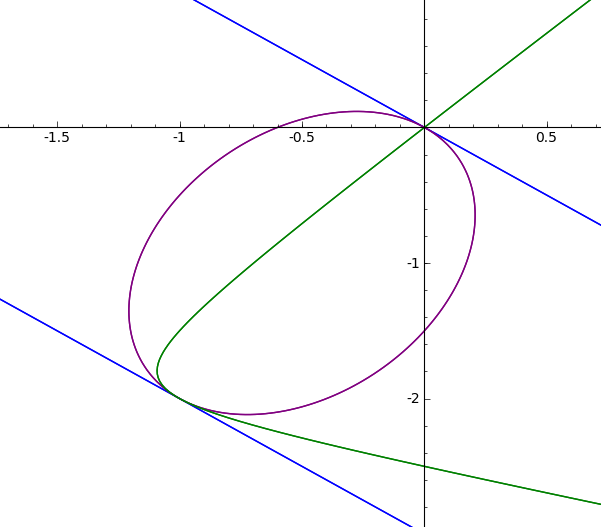}
\caption{The ellipsoidal wall $\cW(\cO_S,\cO_S(-E-2F)[1])$ and the hyperboloidal wall $\cW(\TE,\cO_S(-E-2F)[1])$.}
\label{fig:quiverregion''}
\end{figure}

The projection of the union of the two quiver regions $R_{\bF'}\cup R_{\bF''}$ to the $xy$-plane is the region strictly bounded by the ellipse that is the intersection of $\cW(\cO_S,\cO_S(-E-2F)[1])$ with the $xy$-plane. Indeed, the two hyperboloidal walls $\cW(\cO_S,\TE)$ and $\cW(\TE,\cO_S(-E-2F)[1])$ overlap, as in Figure \ref{fig:quiverregions'''}.

\begin{figure}[h]
\centering
\includegraphics[scale=.4]{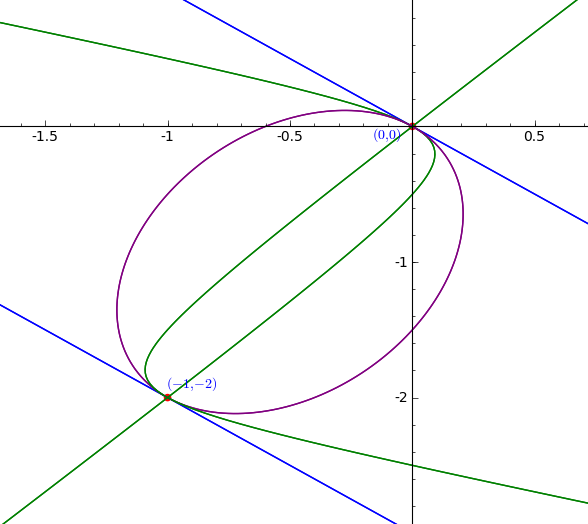}
\caption{The projection of the quiver regions $R_{\bF'}$ and $R_{\bF''}$  is the region strictly bounded by the ellipse $\cW(\cO_S,\cO_S(-E-2F)[1])$.}
\label{fig:quiverregions'''}
\end{figure}

The ellipse bounds a parallelogram with two corners cut off:
$$U=\{(x,y)| -1 \leq x \leq 0,\ x - 1 \leq y \leq x\} \smallsetminus \{(0,0), (-1,-2)\}.$$
The analogous region $U(p,q)$ associated to the quiver regions $R_{\bF'\otimes\cO_S(p,q)}$ and $R_{\bF''\otimes\cO_S(p,q)}$ is $U+(p,q)$, where the $+$ indicates component-wise addition, and together the regions $U(p,q)$ cover the entire $xy$-plane. Recalling the argument given in Section \ref{strategy}, we have shown the following.

\begin{pro}
\label{pro:blp2conj}
Conjecture \ref{con:main} holds for $S=Bl_p\Ptoo$. In particular, if the invariants $v$ satisfy the Bogomolov inequality, then every moduli space $\cM_\sigma(v)$ is isomorphic to a moduli space $\cM_{\sigma'}(v)$ where $\sigma'$ is in a quiver region $R_{\bF'\otimes\cO_S(p,q)}$ or $R_{\bF''\otimes\cO_S(p,q)}$, and hence $\cM_\sigma(v)$ is projective.
\end{pro}


\appendix

%
%
%
%


\section{Mutations defined by height functions}
\label{tiltingquivers}

Here we give the precise definition of a mutation defined by a height function for an object $E_\star$ in an exceptional collection or helix. We first address a few preliminaries before the definition then show that the information necessary to perform a mutation defined by a height function is contained in the quiver associated to a particular thread of the helix. We recall the definitions given in Section $\ref{exccolhts}$, but for the sake of generality we now allow the surface $S$ to be a smooth, projective variety $X$ over $\bC$.

Given an exceptional collection $\bE$ of objects in $D(X)$, we define the \textbf{right orthogonal category} as the full subcategory $\bE^\perp = \{ G \in D\ |\ \Homdot{E_\star}{G}=0 \text{ for all } E_\star\in\bE \}$. The \textbf{left orthogonal category} ${}^\perp\bE$ is defined similarly.  For $E_\star$ an exceptional object and $G\in{}^\perp E_\star$, the \textbf{left mutation}\footnote{Dually, we may define a right mutation. } of $G$ through $E_\star$, $L_{E_\star}(G)$, is defined by the canonical evaluation triangle $$\Homdot{E_\star}{G}\otimes E_\star \To G \To L_{E_\star}(G).$$ Note that $L_{E_\star}(G)\in E_\star^\perp$. If $\bE=(E_1,\ldots,E_k)$ we will denote $L_{E_1} L_{E_2}\cdots L_{E_k} G$ by $L_{\bE}(G)$.



Let $\bE=(E_1,\ldots,E_n)$ be a full strong exceptional collection of objects in $D(X)$ with dual collection $\bF=(F_n,\ldots,F_1)$. Bridgeland and Stern \cite{helicesondelpezzos} define the notion of $p$-related, which is crucial to that of height functions: if $i<j$ we say that $E_i$ and $E_j$ are \textbf{$p$-related} if $\Homx{k}{F_j}{F_i} = 0$ for $k \neq p$. 

A \textbf{height function} for an object $E_\star\in\bE$ is a function (called a``levelling'') $\phi:\bE \To \bZ$ such that $\phi^{-1}(0)=\{E_\star\}$, and $\phi(E_j)=p\neq 0$ implies that $E$ and $E_j$ are $p$-related if $p\geq 0$, or if $p<0$ then $E_j$ and $E$ are $-p$-related. One defines height functions for helices be asking $\phi(E_{i+n})=\phi(E_i)+1 + \dim X$ for each $i$ and the above $p$-related conditions to hold for any thread in the helix.

Given a geometric helix $\bH=(E_i)_{i\in\bZ}$ with a height function $\phi$ for $E_\star$, a \textbf{mutation defined by a height function} for $E_\star$ (henceforth called a \textbf{left tilt} at $E_\star$) constructs a related helix and levelling. We describe the operation algorithmically: to perform a left tilt at $E_\star$...
\begin{itemize}
\item Choose a thread $\bE=(E_k,\ldots,E_{k+n})$ of $\bH$ which contains both $E_\star$ and $\phi^{-1}(-1)=:\bE_{-1}$, i.e where $\bE = (\ldots, \bE_{-1}, E_\star, \ldots)$.
\item Left mutate (and shift) $E$ through $\bE_{-1}$ and keep this new positioning, i.e.  $\bE' = (\ldots, L_{\bE_{-1}}E_\star[-1], \bE_{-1}, \ldots)$.
\item Use $\bE'$ to generate a helix $\bH'$ and a levelling $\phi'$ such that $\phi'^{-1}(-1)=L_{\bE_{-1}}E_\star[-1]$ and $\phi'^{-1}(0)=\Emo$. 
\item The pair ($\bH', \phi')=:\sigma_0(\bH,\phi)$ is the result of the left tilt at $E_\star$.
\end{itemize}

Recall (Section \ref{exccolhts}) that to a full strong exceptional collection we associate a quiver where the number of arrows from vertex $i$ to vertex $j$ is $n_{ij}=\dim \Homx{1}{F_j}{F_i}$. We claim that the helix $\bH'$ obtained via a left tilt at $E_\star$ can be deduced using the the quiver associated to the thread $\bE=(E_1,\ldots,E_n=E_\star)$\footnote{Note that we may always harmlessly change the indexing of $\bH$ so that $E_\star=E_n$.}. The idea is that, if $E_i$'s vertex has an arrow to $E_\star$'s vertex, then $\dim \Homx{1}{F_n}{F_i} > 0 $ and it will follow that $E_i$ and $E_\star$ are 1-related and hence $E_i \in \bE_{-1}$. The essence of Proposition \ref{quivermutprop} is that the other objects in $\bE_{-1}$ do not meaningfully affect the result of the mutation.

\begin{pro}
\label{quivermutprop}
Let $\bH$ be a geometric helix of exceptional objects of $D(X)$ and let $\phi$ be a height function for $E_\star \in \bH$. Consider the thread $\bE = (E_1,\ldots,E_n=E_\star)$ and set 
$$\bE_{-1}^a := \{\text{objects in $\bE$ whose vertex has arrows to $E_\star$'s vertex}\}.$$ 
Then the helix $\bH^\prime$ obtained through a left tilt at $E_\star$ is obtained by taking $E_\star$, moving it just to the left of the last object of $\bE_{-1}^a$, then replacing $E_\star$ with $L_{\bE_{-1}^a}(E_\star)[-1]$ and generating the helix\footnote{We speak algorithmically to avoid cumbersome notation. For instance, a priori there may be objects between $E_\star$ and those of $\Emo^a$.}.

\end{pro} 

\begin{rem}
We identify two helices if they differ by a rearrangement of mutually orthogonal, adjacent objects. Since $A$ mutually orthogonal to $B$ implies that $L_A B = B$, it follows that the objects of the dual collection $\bF$ do not change after shuffling mutually orthogonal objects of a thread $\bE$.
\end{rem}


\begin{proof}[Proof of Proposition \ref{quivermutprop}]
To left tilt $\bH$ at $E_\star$, one takes $E_\star$, moves it just to the left of $\Emo$ and then replaces $E_\star$ with $L_{\Emo} (E_\star)[-1]$. We show that our process produces the same result.

To begin, in the proof of \cite[Proposition 7.1]{helicesondelpezzos} (using $m=0, k=1, i=n$), we may replace $\Emo$ with $\Emo^a$,  since if $E_j$ has no arrows to $E_\star$ then  $d^1_{jn} := \text{dim }\Homx{1}{F_n}{F_j} = 0$. It follows that $L_{\Emo} (E_\star) \cong L_{\Emo^a} (E_\star)=:\cL$. Furthermore, it follows from  \cite[Proposition 7.1]{helicesondelpezzos} that for any two height functions $\phi, \widetilde{\phi}$ for $E_\star$ we have $L_{\Emo}(E_\star)\cong L_{\widetilde{\bE}_{-1}}(E_\star)$.

Call the sequence our process produces $\bH^a$. We show that $\bH^a$ is obtained by a left tilt at $E_\star$ by adjusting the height function $\phi$ to obtain a new height function $\widetilde{\phi}$, where the left-most object of $L_{\widetilde{\bE}_{-1}}(E)$ is the left-most object in $\Emo^a$. It follows that $\bH^a$ is a geometric helix. We then show that $\bH^a$ differs from $\bH'$ by moving $\cL$ across a (possibly empty) set of objects with which it is mutually orthogonal. 

Let $\bF = (F_n,\ldots,F_1)$ be the dual collection corresponding to $\bE$.  Since the height function $\phi$ exists by assumption, if $\Homx{p}{F_n}{F_i} \neq 0$ for some $p$, then $p$ is the only such degree. Now, suppose $E_i \in \Emo$. Then either $\dim \Homx{1}{F_n}{F_i} \neq 0$ (and thus $E_i$'s vertex has an arrow to $E_\star$'s vertex) or $\dim \Homx{p}{F_n}{F_i} = 0$ for all  $p$ (in which case there are no arrows from $E_i$'s vertex to $E_\star$'s vertex).


By definition, $\Emo^a$ is the subset of $\bE$ with $\dim \Homx{1}{F_0}{F_i} \neq 0$ and thus $\Emo^a \subset \Emo$. It follows that any object $E_j$ of $\Emo$ that is to the left of $\Emo^a$ (in the ordering of $\bE$) has $\Homx{p}{F_n}{F_j} = 0$ for all $p$ and so can be moved into $\bE_{-2}$. Doing this gives a height function $\widetilde{\phi}$, where a left tilt at $E_\star$ is accomplished by moving $E_\star$ just to the left of $\widetilde{\bE}_{-1}$ (which is the same as just to the left of $\Emo^a$) and replacing $E_\star$ with $L_{\widetilde{\bE}_{-1}}(E_\star) \cong L_{\Emo}(E_\star) \cong \cL$.


Since placing $\cL$ either to the left of $\Emo$ or to the left of $\Emo^a$ gives a geometric helix, we conclude that $\cL$ is mutually orthogonal to all objects that are in $\Emo$ but to the left of $\Emo^a$.
\end{proof}







\section{Calculations}


\subsection{$\chi(A,B)$ in $\bP^1\times\bP^1$}

Let $A,B$ be two objects of rank $r_A,r_B$, respectively.
Suppose moreover that $c_1(A)=a_1D_1+a_2D_2$ and $c_1(B)=b_1D_1+b_2D_2$.
Then,
$$ \chi(A,B) = r_A r_B + \frac12 (r_A d_B - r_B d_A) + r_B \ch_2(A) + r_A \ch_2(B) - c_1(A) c_1(B), $$
where $d_A$ (resp.\ $d_B$) is the degree of $A$ (resp.\ $B$) defined by $d_A=-K_S\cdot c_1(A)$ (resp.\ $d_B=-K_S\cdot c_1(B)$).
Since $K_S=-2D_1-2D_2$, we have that $d_A=2(a_1+a_2)$ (resp.\ $d_B=2(b_1+b_2)$), and
$$ \chi(A,B) = r_A r_B + r_A (b_1 + b_2) - r_B (a_1 + a_2) + r_B \ch_2(A) + r_A \ch_2(B) - c_1(A) c_1(B). $$


\subsection{$L_{\cO_S(1,0)}L_{\cO_S(0,1)}(\cO_S(1,1))$ in $\bP^1\times\bP^1$}\label{E'}

\begin{itemize}
\item
To calculate $L_{\cO_S(0,1)}(\cO_S(1,1))$, note that
$$ \chi(\cO_S(0,1),\cO_S(1,1)) = 1 + 2 - 1 + 1 - 1 = 2. $$
Therefore, we obtain the map
$$ \cO_S(0,1)^{\oplus2} \lfun \cO_S(1,1), $$
and $L_{\cO_S(0,1)}(\cO_S(1,1))=\cO_S(-1,1)[1]$.
\begin{itemize}
\item
To calculate $L_{\cO_S(1,0)}(\cO_S(-1,1)[1])$, note that
$$ \chi(\cO_S(1,0),\cO_S(-1,1)) = 1 - 1 - 1 - 1 = -2. $$
Therefore, we obtain the map
$$ \cO_S(1,0)^{\oplus2} \lfun \cO_S(-1,1)[1], $$
and $L_{\cO_S(1,0)}(\cO_S(-1,1)[1])=G[1]$, where $G$ is a rank 3 extension of $\cO_S(1,0)^{\oplus2}$ by $\cO_S(-1,1)$.
We have that $\rk(G)=3$, $c_1(G)=D_1+D_2$, and $\ch_2(G)=-1$.
\end{itemize}

\end{itemize}


\subsection{$\bF'$ in $\bP^1\times\bP^1$}\label{F'}

Recall that
$$ \bE' = ( \cO_S, G, \cO_S(1,0), \cO_S(0,1) ), $$
where $G$ is a sheaf of rank 3 with $c_1(G)=D_1+D_2$ and $\ch_2(G)=-1$.

We calculate here the dual collection $\bF'$.

\begin{itemize}
\item
To calculate $L_{\cO_S}(G)$, note that
$$ \chi(\cO_S,G) = 3 + 2 - 1 = 4. $$
Therefore, we obtain the map
$$ \cO_S^{\oplus4} \lfun G, $$
and $L_{\cO_S}(G)=\cO_S(-1,-1)[1]$.
\item
To calculate $L_{G}(\cO_S(1,0))$, note that
$$ \chi(G,\cO_S(1,0)) = 3 + 3 - 2 - 1 - 1 = 2. $$
Therefore, we obtain the map
$$ G^{\oplus2} \lfun \cO_S(1,0), $$
and $L_{G}(\cO_S(1,0))=G_1[1]$,
where $G_1$ has $\rk(G_1)=5$, $c_1(G_1)=D_1+2D_2$, and $\ch_2(G_1)=-2$.
\begin{itemize}
\item
To calculate $L_{\cO_S}(G_1[1])$, note that
$$ \chi(\cO_S,G_1) = 5 + 3 - 2 = 6. $$
Therefore, we obtain the map
$$ \cO_S^{\oplus6}[1] \lfun G_1[1], $$
and $L_{\cO_S}(G_1[1])=\cO_S(-1,-2)[2]$.
\end{itemize}
\item
We may calculate $L_{\cO_S} L_{G} L_{\cO_S(1,0)} (\cO_S(0,1))$ with similar computations, but it is faster to use \cite[Corollary 2.10 \& Definition 3.1]{helicesondelpezzos} to obtain $$L_{\cO_S} L_{G} L_{\cO_S(1,0)} (\cO_S(0,1)) = \cO_S(0,1) \otimes \omega_S[2] = \cO_S(-2,-1)[2].$$

\end{itemize}

Therefore,
$$ \bF' = ( \cO_S(-2,-1)[2], \cO_S(-1,-2)[2], \cO_S(-1,-1)[1], \cO_S ). $$


\subsection{$\chi(A,B)$ in $Bl_p\bP^2$}

Let $A,B$ be two objects of rank $r_A,r_B$, respectively.
Suppose moreover that $c_1(A)=a_EE+a_FF$ and $c_1(B)=b_EE+b_FF$.
Then,
$$ \chi(A,B) = r_A r_B + \frac12 (r_A d_B - r_B d_A) + r_B \ch_2(A) + r_A \ch_2(B) - c_1(A) c_1(B), $$
where $d_A$ (resp.\ $d_B$) is the degree of $A$ (resp.\ $B$) defined by $d_A=-K_S\cdot c_1(A)$ (resp.\ $d_B=-K_S\cdot c_1(B)$).
Since $K_S=-3H+E=-2E-3F$, we have that $d_A=a_E+2a_F$ (resp.\ $d_B=b_E+2b_F$).

In particular,
$$ \chi(\cO_S,B) = r_B + \frac12 b_E + b_F + \ch_2(B). $$
If $B$ is a line bundle $L$, then $r_B=1$ and $\ch_2(B)=(c_1(B))^2/2=b_Eb_F-b_E^2/2$, so we would have
$$ \chi(\cO_S,L) = 1 + \frac12 (b_E - b_E^2) + b_F + b_E b_F. $$


\subsection{$L_{\cO_S(F)}L_{\cO_S(E+F)}(\cO_S(E+2F))$ in $Bl_p\bP^2$}\label{E'2}

\begin{itemize}
\item
To calculate $L_{\cO_S(E+F)}(\cO_S(E+2F))$, note that
$$ \chi(\cO_S(E+F),\cO_S(E+2F)) = 1 + \frac12 (5 - 3) + \frac12 + \frac32 - 2 = 2. $$
Therefore, we obtain the map
$$ \cO_S(E+F)^{\oplus2} \lfun \cO_S(E+2F), $$
and $L_{\cO_S(E+F)}(\cO_S(E+2F))=\cO_S(E)[1]$.
\begin{itemize}
\item
To calculate $L_{\cO_S(F)}(\cO_S(E)[1])$, note that
$$ \chi(\cO_S(F),\cO_S(E)) = 1 + \frac12 (1 - 2) + 0 - \frac12 - 1 = -1. $$
Therefore, we obtain the map
$$ \cO_S(F) \lfun \cO_S(E)[1], $$
and $L_{\cO_S(F)}(\cO_S(E)[1])=G_1[1]$, where $G_1$ is the rank 2 extension of $\cO_S(F)$ by $\cO_S(E)$.
We have that $\rk(G_1)=2$, $c_1(G_1)=E+F$, and $\ch_2(G_1)=-1/2$.
\end{itemize}
\end{itemize}


\subsection{$\bF'$ in $Bl_p\bP^2$}\label{F'2}

Recall that
$$ \bE' = ( \cO_S, G_1, \cO_S(F), \cO_S(E+F) ), $$
where $G_1$ is a sheaf of rank 2 with $c_1(G_1)=E+F$, and $\ch_2(G_1)=-1/2$.
We calculate here the dual collection $\bF'$.

\begin{itemize}
\item
To calculate $L_{\cO_S}(G_1)$, note that
$$ \chi(\cO_S,G_1) = 2 + \frac12 \cdot 3 - \frac12 = 3. $$
Therefore, we obtain the map
$$ \cO_S^{\oplus3} \lfun G_1, $$
and $L_{\cO_S}(G_1)=\cO_S(-E-F)[1]$.
\item
To calculate $L_{G_1}(\cO_S(F))$, note that
$$ \chi(G_1,\cO_S(F)) = 2 + \frac12 (4 - 3) - \frac12 - 1 = 1. $$
Therefore, we obtain the map
$$ G_1 \lfun \cO_S(F), $$
and $L_{G_1}(\cO_S(F))=\cO_S(E)[1]$.
\begin{itemize}
\item
To calculate $L_{\cO_S}(\cO_S(E)[1])$, note that
$$ \chi(\cO_S,\cO_S(E)) = 1 + \frac12 - \frac12 = 1. $$
Therefore, we obtain the map
$$ \cO_S[1] \lfun \cO_S(E)[1], $$
and $L_{\cO_S}(\cO_S(E)[1])=\cO_S(E)|_E[1]$.
Note that $\cO_S(E)|_E$ has rank 0, $c_1=E$, and $\ch_2=-1/2$.
It is therefore isomorphic to $\cO_S(-E-2F)|_E$, which has the same invariants.
\end{itemize}
\item
Using \cite[Corollary 2.10 \& Definition 3.1]{helicesondelpezzos}, we calculate $$L_{\cO_S} L_{G_1} L_{\cO_S(F)} (\cO_S(E+F)) = \cO_S(E+F) \otimes \omega_S[2] = \cO_S(-E-2F)[2].$$

\end{itemize}

Therefore,
$$ \bF' = ( \cO_S(-E-2F)[2], \cO_S(E)|_E[1], \cO_S(-E-F)[1], \cO_S ). $$


\subsection{$L_{\cO_S(E+F)}(\cO_S(2E+2F))$ in $Bl_p\bP^2$}\label{E''}

Note that
$$ \chi(\cO_S(E+F),\cO_S(2E+2F)) = 1 + \frac12 (6 - 3) + \frac12 + 2 - 2 = 3. $$
Therefore, we obtain the map
$$ \cO_S(E+F)^{\oplus3} \lfun \cO_S(2E+2F), $$
and $L_{\cO_S(E+F)}(\cO_S(E+2F))=G_2[1]$, where $G_2$ is a sheaf with $\rk(G_2)=2$, $c_1(G_2)=E+F$, and $\ch_2(G_2)=-1/2$.


\subsection{$\bF''$ in $Bl_p\bP^2$}\label{F''}

Recall that
$$ \bE'' = \< \cO_S, \cO_S(E), G_2, \cO_S(E+F) \>, $$
where $G_2$ is a sheaf of rank 2 with $c_1(G_2)=E+F$ and $\ch_2(G_2)=-1/2$.
We calculate here the dual collection $\bF''$.

\begin{itemize}
\item
To calculate $L_{\cO_S}(\cO_S(E))$, note that
$$ \chi(\cO_S,\cO_S(E)) = 1 + \frac12 - \frac12 = 1. $$
Therefore, we obtain the map
$$ \cO_S \lfun \cO_S(E), $$
and $L_{\cO_S}(\cO_S(E))=\TE$.

\item
To calculate $L_{\cO_S(E)}(G_2)$, note that
$$ \chi(\cO_S(E),G_2) = 2 + \frac12 (3 - 2) + 2 \cdot \left(- \frac12\right) - \frac12 = 1. $$
Therefore, we obtain the map
$$ \cO_S(E) \lfun G_2, $$
and $L_G(\cO_S(F))=\cO_S(F)$.
\begin{itemize}
\item
To calculate $L_{\cO_S}(\cO_S(F))$, note that
$$ \chi(\cO_S,\cO_S(F)) = 1 + \frac12 \cdot 2 = 2. $$
Therefore, we obtain the map
$$ \cO_S^{\oplus2} \lfun \cO_S(F), $$
and $L_{\cO_S}(\cO_S(F))=\cO_S(-F)[1]$.
\end{itemize}
\item Using \cite[Corollary 2.10 \& Definition 3.1]{helicesondelpezzos}, we calculate $$L_{\cO_S} L_{\cO_S(E)} L_{G_2} (\cO_S(E+F)) = \cO_S(E+F) \otimes \omega_S[2] = \cO_S(-E-2F)[2].$$

\end{itemize}

Therefore,
$$ \bF'' = ( \cO_S(-E-2F)[2], \cO_S(-F)[1], \cO_S(E)|_E, \cO_S ). $$

\bibliographystyle{alphaurl}

\bibliography{references}

\begin{thebibliography}{ABCH13}

\bibitem[AB13]{ABL}
Daniele Arcara and Aaron Bertram.
\newblock Bridgeland-stable moduli spaces for {K}-trivial surfaces.
\newblock {\em JEMS}, 15(1):1--38, 2013.
\newblock (with an appendix by Max Lieblich).

\bibitem[ABCH13]{ABCH}
Daniele Arcara, Aaron Bertram, Izzet Coskun, and Jack Huizenga.
\newblock The minimal model program for the {H}ilbert scheme of points on
  $\mathbb{P}^2$ and {B}ridgeland stability.
\newblock {\em Adv. Math.}, 235:580--626, 2013.

\bibitem[AM14]{linebundlessurfs}
Daniele Arcara and Eric Miles.
\newblock Bridgland stability of line bundles on surfaces.
\newblock {\em Preprint}, 2014.
\newblock arXiv:1401.6149.

\bibitem[BM14]{BMprojandbirat}
Arend Bayer and Emanuele Macr{\`{\i}}.
\newblock Projectivity and birational geometry of {B}ridgeland moduli spaces.
\newblock {\em J. Amer. Math. Soc.}, 27(3):707--752, 2014.

\bibitem[Bri07]{stabcondsontricats}
Tom Bridgeland.
\newblock Stability conditions on triangulated cateogories.
\newblock {\em Ann. Math.}, 166:317--345, 2007.

\bibitem[BS10]{helicesondelpezzos}
Tom Bridgeland and David Stern.
\newblock Helices on del {P}ezzo surfaces and tilting {C}alabi-{Y}au algebras.
\newblock {\em Adv. Math.}, 224(4):1672--1716, 2010.

\bibitem[GK04]{helixtheory}
A.~L. Gorodentsev and S.~A. Kuleshov.
\newblock Helix theory.
\newblock {\em Mosc. Math. J.}, 4(2):377--440, 535, 2004.

\bibitem[Kin94]{kingquiver}
A.D. King.
\newblock Moduli of representations of finite dimensional algebras.
\newblock {\em Quart. J. Math. Oxford}, 45:515--530, 1994.

\bibitem[Mac07]{macri}
Emanuele Macr\`i.
\newblock Stability conditions on curves.
\newblock {\em Mathematical Research Letters}, 14:657--672, 2007.

\bibitem[Mac14]{MaciociaWalls}
Antony Maciocia.
\newblock Computing the walls associated to {B}ridgeland stability conditions
  on projective surfaces.
\newblock {\em Asian J. Math.}, 18(2), 2014.

\bibitem[MM13]{MacMea}
Antony Maciocia and Ciaran Meachan.
\newblock Rank one bridgeland stable moduli spaces on a principally polarized
  abelian surface.
\newblock {\em IMRN}, 2013(9):2054--2077, 2013.

\bibitem[Tod13]{stabcondsextrcontrs}
Yukinobu Toda.
\newblock Stability conditions and extremal contractions.
\newblock {\em Math. Ann.}, 357(2):631--685, 2013.

\end{thebibliography}

\end{document}